\documentclass[a4paper,10pt]{article}

\usepackage[hidelinks]{hyperref}

\usepackage{a4wide}
\usepackage{amsmath,amssymb,amsthm,xspace,float}
\usepackage{amsfonts}
\usepackage{color}
\usepackage{amsmath, amssymb, amsthm}
\usepackage{tikz}
\usepackage{mathtools}
\usepackage{bm}
\usepackage{url}

\usepackage{enumitem}

\newtheorem{thm}{Theorem}
\newtheorem{definition}[thm]{Definition}
\newtheorem{lemma}[thm]{Lemma}
\newtheorem{proposition}[thm]{Proposition}

\newtheorem{con}[thm]{Conjecture}

\theoremstyle{definition}
\newtheorem{remark}[thm]{Remark}

\newcommand{\callanpat}{1\text{-}{23}\text{-}4}
\newcommand{\pat}{1\text{-}34\text{-}2}

\title{Enumerating five families of pattern-avoiding inversion sequences; and introducing the powered Catalan numbers}
\author{Nicholas R. Beaton, Mathilde Bouvel, Veronica Guerrini, and Simone Rinaldi}
\date{\today}

\begin{document}

\maketitle

\begin{abstract} 
The first problem addressed by this article is the enumeration of some families of pattern-avoiding inversion sequences. 
We solve some enumerative conjectures left open by the foundational work on the topics by Corteel et al., some of these being also solved independently by Lin, and Kim and Lin. 
The strength of our approach is its robustness: 
we enumerate four families $F_1 \subset F_2 \subset F_3 \subset F_4$ of pattern-avoiding inversion sequences ordered by inclusion using the same approach. 
More precisely, we provide a generating tree (with associated succession rule) for each family $F_i$ which generalizes the one for the family $F_{i-1}$. 

The second topic of the paper is the enumeration of a fifth family $F_5$ of pattern-avoiding inversion sequences (containing $F_4$). 
This enumeration is also solved \emph{via} a succession rule, which however does not generalize the one for $F_4$. 
The associated enumeration sequence, which we call the \emph{powered Catalan numbers}, is quite intriguing, and further investigated. 
We provide two different succession rules for it, denoted $\Omega_{pCat}$ and $\Omega_{steady}$, 
and show that they define two types of families enumerated by powered Catalan numbers. 
Among such families, we introduce the \emph{steady paths}, which are naturally associated with $\Omega_{steady}$. 
They allow us to bridge the gap between the two types of families enumerated by powered Catalan numbers: 
indeed, we provide a size-preserving bijection between steady paths and valley-marked Dyck paths (which are naturally associated with $\Omega_{pCat}$). 

Along the way, we provide several nice connections to families of permutations defined by the avoidance of vincular patterns, and some enumerative conjectures. 
\end{abstract}

\section{Introduction and preliminaries}

\subsection{Context of our work}
An {\em inversion sequence} of length $n$ is any integer sequence $(e_1, \ldots, e_n)$ satisfying $0\leq e_i <i$, for all $i=1, \ldots ,n$. 
There is a well-known bijection $\mathtt{T}: S_n \to I_n$ between the set $S_n$ of all permutations of length (or size) $n$ and the set $I_n$ of all inversion sequences of length $n$, which maps a permutation $\pi\in \mathcal{S}_n$ into its {\em left inversion table} $(t_1,\ldots,t_n)$, where $t_i=|\{j:j>i\mbox{ and }\pi_i>\pi_j\}|$. 
This bijection is actually at the origin of the name inversion sequences.  

The study of pattern-containment or pattern-avoidance in inversion sequences was first introduced in~\cite{mansour}, and then further investigated in~\cite{corteel}. 
Namely, in \cite{mansour}, Mansour and Shattuck studied inversion sequences that avoid permutations of length $3$, 
while in~\cite{corteel}, Corteel et al. proposed the study of inversion sequences avoiding subwords of length $3$. 
The definition of inversion sequences avoiding words (which may in addition be permutations) is straightforward: 
for instance, the inversion sequences that avoid the word $110$ (resp. the permutation $132$) are those with no $i<j<k$ such that $e_i=e_j>e_k$ (resp. $e_i<e_k<e_j$). 
Pattern-avoidance on special families of inversion sequences has also been studied in the literature, 
namely by Duncan and Steingr\'imsson on {\em ascent sequences} -- see~\cite{duncan}. 
 
The pattern-avoiding inversion sequences of~\cite{corteel} were further generalized in \cite{savage}, 
extending the notion of pattern-avoidance to triples of binary relations $(\rho_1, \rho_2,\rho_3)$. 
More precisely, they denote by $\mathbf{I}_n(\rho_1, \rho_2,\rho_3)$ the set of all inversion sequences in $I_n$ having no three indices $i<j<k$ such that $e_i \rho_1 e_j$, $e_j \rho_2 e_k$, and $e_i \rho_3 e_k$, 
and by $\mathbf{I}(\rho_1, \rho_2,\rho_3)= \cup_n \mathbf{I}_n(\rho_1, \rho_2,\rho_3)$. 
For example, the sets $\mathbf{I}_n(=,>,>)$ and $\mathbf{I}_n(110)$ coincide for every $n$. 
In \cite{savage} all triples of relations in $\{ <,>,\leq, \geq, =, \neq, - \}^3$ are considered, where ``$-$'' stands for any possible relation on a set $S$, {\em i.e.} $x-y$ for any $(x,y) \in S \times S$. 
Therefore, all the $343$ possible triples of relations are examined and the resulting families of pattern-avoiding inversion sequences are subdivided into $98$ equivalence classes. 
Many enumeration results complementing those in~\cite{corteel,mansour} have been found in~\cite{savage}. 
In addition, several conjectures have been formulated in~\cite{savage}. 
Some (but by far not all!) of them have been proved between the moment a first version of~\cite{savage} was posted on the arXiv and its publication, 
and references to these recent proofs can also be found in the published version of~\cite{savage}. 

In this paper we study five families of inversion sequences which form a hierarchy for the inclusion order. 
The enumeration of these classes -- by well-known sequences, such as those of the Catalan, the Baxter, and the newly introduced semi-Baxter numbers~\cite{semibaxterlong} -- 
was originally conjectured in the first version of~\cite{savage}. 
These conjectures have attracted the attention of a fair number of combinatorialists, 
resulting in proofs for all of them, independently of our paper. 
Still, our work reproves these enumeration results. 
Along the way, we further try to establish bijective correspondences between these families of inversion sequences and other known combinatorial structures. 
The most remarkable feature of our work is that all the families of inversion sequences are presented and studied in a unified way by means of {\em generating trees}. 
Before proceeding, let us briefly recall some basics about generating trees. Details can be found for instance in~\cite{GFGT,Eco,BM,West_gt}. 

\subsection{Basics of generating trees}

Consider a combinatorial class $\mathcal{C}$, that is to say a set of discrete objects equipped with a notion of size such that the number of objects of size $n$ is finite, for any $n$. We assume also that $\mathcal{C}$ contains exactly one object of size $1$. 
A \emph{generating tree} for $\mathcal{C}$ is an infinite rooted tree whose vertices are the objects of $\mathcal{C}$ each appearing exactly once in the tree, 
and such that objects of size $n$ are at level $n$ (with the convention that the root is at level $1$). 
The children of some object $c \in \mathcal{C}$ are obtained by adding an \emph{atom} (\emph{i.e.}~a piece of the object that makes its size increase by $1$) to $c$. 
Since every object appears only once in the generating tree, not all possible additions are acceptable. 
We enforce the unique appearance property by considering only additions that follow some prescribed rules and call the \emph{growth} of $\mathcal{C}$ the process of adding atoms according to these rules. 

To illustrate these definitions, we describe the classical growth for the family of Dyck paths, as given by~\cite{Eco}. 
Recall that a Dyck path of semi-length $n$ is a lattice path using up $U=(1,1)$ and down $D=(1,-1)$ unit steps, running from $(0,0)$  to $(2n,0)$ and remaining weakly above the $x$-axis.  
The atoms we consider are $UD$ factors, a.k.a. \emph{peaks}, which are added to a given Dyck path. 
To ensure that all Dyck paths appear exactly once in the generating tree, a peak is inserted only in a point of the \emph{last descent}, defined as the longest suffix containing only $D$ letters. 
More precisely, the children of the Dyck path $w\,UD^k$ are
$w\, U\bm{UD}D^k$, $w\, UD\bm{UD}D^{k-1}$,\dots, $w\, UD^{k-1}\bm{UD}D$, $w\,UD^k\bm{UD}$.

The first few levels of the generating tree for Dyck paths are shown in Figure~\ref{fig:gen_tree} (left). 

\begin{figure}[ht]
\begin{center}
\includegraphics[scale=0.4]{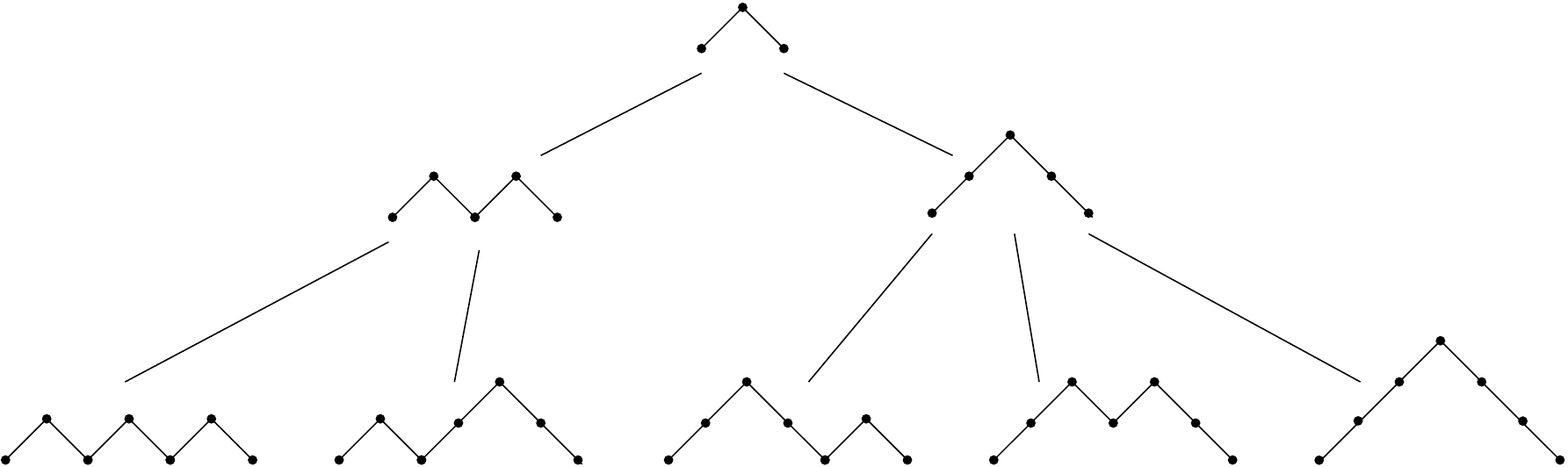} \qquad \includegraphics[scale=0.45]{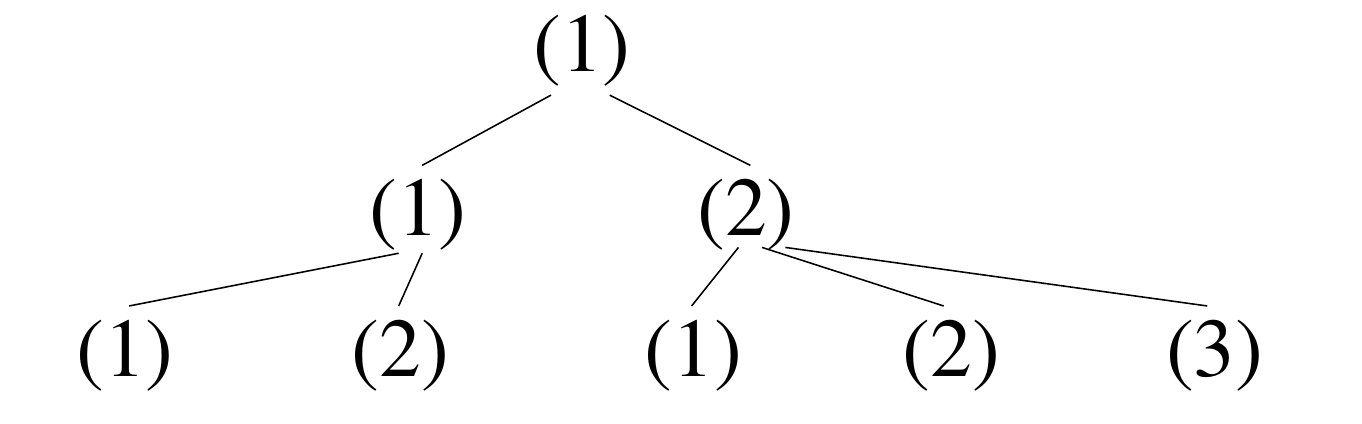}
\end{center}\vspace{-3mm}
\caption{Two ways of looking at the generating tree for Dyck paths: with objects (left) and with labels from the succession rule $\Omega_{Cat}$ (right).\label{fig:gen_tree}}
\end{figure}

When the growth of $\mathcal{C}$ is particularly regular, we encapsulate it in a {\em succession rule}. 
This applies more precisely when there exist statistics whose evaluations control the number of objects produced in the generating tree.
A succession rule consists of one starting label (\emph{axiom}) corresponding to the value of the statistics on the root object 
and of a \emph{set of productions} encoding the way in which these evaluations spread in the generating tree -- see Figure~\ref{fig:gen_tree}(right). 
The growth of Dyck paths presented earlier is governed by the statistic ``length of the last descent'', 
so that it corresponds to the following succession rule, where each label $(k)$ indicates the number of $D$ steps of the last descent in a Dyck path,
\[\Omega_{Cat}=\left\{\begin{array}{ll}
(1)\\
\\
(k) &\rightsquigarrow (1),(2), \dots , (k),(k+1). \end{array}\right. \label{page:OmegaCat}\]

Obviously, as we discuss in~\cite{paper1}, the sequence enumerating the class $\mathcal{C}$ can be recovered from the succession rule itself, 
without reference to the specifics of the objects in $\mathcal{C}$: 
indeed, the $n$th term of the sequence is the total number of labels (counted with repetition) that are produced from the root by $n-1$ applications of the set of productions, 
or equivalently, the number of nodes at level $n$ in the generating tree. 
For instance, the well-know fact that Dyck paths are counted by Catalan numbers (sequence A000108 in \cite{OEIS}) can be recovered by counting nodes at each level $n$ in the above generating tree. 

\subsection{Content of the paper}

In our study we focus on five different families of pattern-avoiding inversion sequences,  
which are depicted in Figure~\ref{fig:hierarchy}. 
As the figure shows, these families are naturally ordered by inclusion, 
and are enumerated by well-known number sequences.

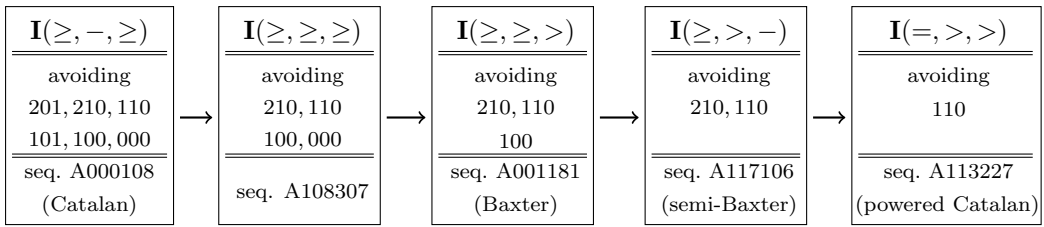
\begin{figure}[ht!]
\centering
\begin{tikzpicture}[scale=.85]
\node at (0.1,-.7) {\footnotesize avoiding};
\node at (0.1,-1.7) {\footnotesize $101, 100,000$};
\node at (0.1,-1.2) {\footnotesize $201,210, 110$};
\draw (-1.2,-3) rectangle (1.4,.4);
\node at (0.1,0) {$\mathbf{I}(\geq, -, \geq)$ };
\draw (-1.1,-0.3) --(1.3,-0.3);
\draw (-1.1,-0.35) --(1.3,-0.35);
\draw (-1.1,-1.9) --(1.3,-1.9);
\draw (-1.1,-1.95) --(1.3,-1.95);
\node at (0.1,-2.2) {\footnotesize  seq.~A000108};
\node at (0.1,-2.7) {\footnotesize  (Catalan)};
\draw[thick,->] (1.5,-1.33)--(2,-1.33);
\node at (3.4,0) {$\mathbf{I}(\geq, \geq, \geq)$ };
\draw (2.1,-3) rectangle (4.6,.4);
\node at (3.4,-.7) {\footnotesize avoiding};
\node at (3.4,-1.7) {\footnotesize $100, 000$};
\node at (3.4,-1.2) {\footnotesize $210, 110$};
\draw (2.2,-0.3) --(4.5,-0.3);
\draw (2.2,-0.35) --(4.5,-0.35);
\draw (2.2,-1.9) --(4.5,-1.9);
\draw (2.2,-1.95) --(4.5,-1.95);
\node at (3.4,-2.5) {\footnotesize seq. A108307};
\draw[thick,->] (4.7,-1.33)--(5.3,-1.33);
\node at (6.7,0) {$\mathbf{I}(\geq, \geq,>)$ };
\draw (5.4,-3) rectangle (7.9,.4);
\node at (6.7,-.7) {\footnotesize avoiding};
\node at (6.7,-1.7) {\footnotesize $100$};
\node at (6.7,-1.2) {\footnotesize $210, 110$};
\draw (5.5,-0.3) --(7.8,-0.3);
\draw (5.5,-0.35) --(7.8,-0.35);
\draw (5.5,-1.9) --(7.8,-1.9);
\draw (5.5,-1.95) --(7.8,-1.95);
\node at (6.7,-2.2) {\footnotesize  seq.~A001181};
\node at (6.7,-2.7) {\footnotesize  (Baxter)};
\draw[thick,->] (8,-1.33)--(8.6,-1.33);
\node at (10,0) {$\mathbf{I}(\geq, >, -)$ };
\draw (8.7,-3) rectangle (11.2,.4);
\node at (10,-.7) {\footnotesize avoiding};
\node at (10,-1.2) {\footnotesize $210, 110$};
\draw (8.8,-0.3) --(11.1,-0.3);
\draw (8.8,-0.35) --(11.1,-0.35);
\draw (8.8,-1.9) --(11.1,-1.9);
\draw (8.8,-1.95) --(11.1,-1.95);
\node at (10,-2.2) {\footnotesize  seq.~A117106};
\node at (10,-2.7) {\footnotesize  (semi-Baxter)};
\draw[thick,->] (11.3,-1.33)--(11.8,-1.33);
\node at (13.4,0) {$\mathbf{I}(=, >, >)$ };
\node at (13.4,-.7) {\footnotesize avoiding};
\node at (13.4,-1.2) {\footnotesize $110$};
\draw (11.9,-3) rectangle (14.83,.4);
\draw (12,-0.3) --(14.7,-0.3);
\draw (12,-0.35) --(14.7,-0.35);
\draw (12,-1.9) --(14.7,-1.9);
\draw (12,-1.95) --(14.7,-1.95);
\node at (13.4,-2.2) {\footnotesize  seq.~A113227};
\node at (13.38,-2.7) {\footnotesize  (powered Catalan)};
\end{tikzpicture}
\caption{A chain of families of inversion sequences ordered by inclusion, with their characterization in terms of pattern avoidance, and their enumerative sequence.\label{fig:hierarchy}}
\end{figure}

The objective of our study is twofold. 
On the one hand we provide (and/or collect) enumerative results about the families of inversion sequences of Figure~\ref{fig:hierarchy}. 
On the other hand we aim at treating all these families in a unified way. 
More precisely, in each of the following sections we first provide a simple combinatorial characterization for the corresponding family of inversion sequences, and then we show a recursive growth that yields a succession rule. 

The main noticeable property of the succession rules provided in Sections~\ref{sec:catinvseq}, \ref{sec:invseq3cross}, \ref{sec:invseq3rd}, and~\ref{sec:invseq4th} 
is that they reveal the hierarchy of Figure~\ref{fig:hierarchy} at the abstract level of succession rules. 
Specifically, the recursive construction (or growth) provided for each family is obtained by extending the construction of the immediately smaller family.
Moreover, the ways in which these growths are encoded by labels in succession rules are also each a natural extension of the case of the immediately smaller family. 
Hence, these examples provide another illustration of the idea of generalizing/specializing succession rules that we discussed in details in~\cite[Section 2.2]{slicings}. 
The outcome of the discussion in~\cite[Section 2.2]{slicings} is the following proposed definition for generalization/specialization of succession rules. 
To say that a succession rule $\Omega_{\mathcal{B}}$ specializes $\Omega_{\mathcal{A}}$ (equivalently, that $\Omega_{\mathcal{A}}$ generalizes or extends $\Omega_{\mathcal{B}}$), 
we require
\begin{itemize}
\item[(1)] the existence of a comparison relation ``smaller than or equal to'' 
between the labels of $\Omega_{\mathcal{B}}$ and those of $\Omega_{\mathcal{A}}$, 
and,
\item[(2)] for any labels $\ell_A$ of $\Omega_{\mathcal{A}}$ and $\ell_B$ of $\Omega_{\mathcal{B}}$ with $\ell_B$ smaller than or equal to $\ell_A$, 
a way of mapping the productions of the label $\ell_B$ in $\Omega_{\mathcal{B}}$ to a subset of the productions of the label $\ell_A$ in $\Omega_{\mathcal{A}}$, 
such that a label is always mapped to a larger or equal one.
\end{itemize}
Comparing Propositions~\ref{prop:catinvseq}, \ref{prop:3crosinvseq}, \ref{prop:baxinvseq} and \ref{prop:finalinvseq}, 
and mapping the labels in the obvious way, it is easy to see that the succession rules in these propositions 
satisfy this proposed definition (the comparison relation being here just the componentwise natural order on integers). 

\medskip

We conclude our introduction with a few words commenting on the classes of our hierarchy and our results on them.

\begin{itemize}
\item[i)] We start in Section~\ref{sec:catinvseq} with $\mathbf{I}(\geq, -, \geq)$, which we call the family of Catalan inversion sequences. 
We define two recursive growths for this family, one according to $\Omega_{Cat}$ (hence proving that $\mathbf{I}(\geq, -, \geq)$ is enumerated by the Catalan numbers) 
and a second one that turns out to be a new succession rule for the Catalan numbers. 
The fact that this family of inversion sequences is enumerated by the Catalan numbers was conjectured in~\cite{savage} and it has recently been proved independently of us by Kim and Lin in~\cite{KL}. 
Moreover, we are able to relate the family of Catalan inversion sequences to a family of permutations defined by the avoidance of {\em vincular patterns}, 
proving that they are in bijection with a family of pattern-avoiding permutations.
\item[ii)]  In Section~\ref{sec:invseq3cross} we consider the family $\mathbf{I}(\geq,\geq,\geq)$. 
This class has been considered independently of us by Lin in the article~\cite{LinKernel}, 
which proves the conjecture (originally formulated in~\cite{savage}) that these inversion sequences are counted by sequence A108307 on~\cite{OEIS} -- 
defining the enumerative sequence of set partitions of $\{1,\ldots, n\}$ that avoid enhanced 3-crossings~\cite{MBM_Xin}. 
We review Lin's proof, which fits perfectly in the hierarchy of succession rules that we present. 
\item[iii)] In Section~\ref{sec:invseq3rd} we study inversion sequences in $\mathbf{I}(\geq,\geq,>)$, which we call Baxter inversion sequences. 
This family of inversion sequences was originally conjectured in~\cite{savage} to be counted by Baxter numbers. 
The proof of this conjecture was provided in~\cite{KL} by means of a growth for Baxter inversion sequences that neatly generalizes the previous growth for the family $\mathbf{I}(\geq,\geq,\geq)$.
\item[iv)] In Section~\ref{sec:invseq4th}, we deal with the family $\mathbf{I}(\geq,>,-)$, which we call semi-Baxter inversion sequences. 
Indeed, this family of inversion sequences was originally conjectured in~\cite{savage} to be counted by the sequence A117106~\cite{OEIS}; 
these numbers have been thoroughly studied and named semi-Baxter in the article~\cite{semibaxterlong}, 
which among other results proves this conjecture of~\cite{savage}. 
\item[v)] Finally, in Section~\ref{sec:callaninvseq} we deal with $\mathbf{I}(=,>,>)$, which is the rightmost element of the chain of Figure~\ref{fig:hierarchy}. 
We call the elements of $\mathbf{I}(=,>,>)$ {\em powered Catalan inversion sequences}, 
since the succession rule we provide for them is a ``powered version'' of the classical Catalan succession rule. 
\end{itemize}
When turning to powered Catalan inversion sequences, the hierarchy of Figure~\ref{fig:hierarchy} is broken at the level of succession rules. 
Indeed, although the combinatorial characterization of these objects generalizes naturally that of semi-Baxter inversion sequences, 
we do not have a growth for powered Catalan inversion sequences that generalizes the one of semi-Baxter inversion sequences. 
This motivates the second part of the paper, devoted to the study of this ``powered Catalan'' enumerative sequence from Section~\ref{sec:callaninvseq} on. 
\smallskip

The enumeration of powered Catalan inversion sequences (by A113227,~\cite{OEIS}) was already solved in~\cite{corteel}. 
Our first contribution (in Section~\ref{sec:callaninvseq}) is to prove that they grow according to the succession rule $\Omega_{pCat}$, 
which generalizes the classical rule $\Omega_{Cat}$ by introducing powers in it. 
This motivates the name \emph{powered Catalan numbers} which we have coined for the numbers of sequence A113227. 

Many combinatorial families are enumerated by powered Catalan numbers. Some are presented in Section~\ref{sec:callan}.
These families somehow fall into two categories. Inside each category, the objects seem to be in rather natural bijective correspondence. 
However, between the two categories, the bijections are much less clear. 
Our result of Section~\ref{sec:callan} is to provide a second succession rule for powered Catalan numbers (more precisely, for permutations avoiding the vincular pattern $\callanpat$), 
which should govern the growth of objects in one of these two categories, the other category being naturally associated with the rule $\Omega_{pCat}$. 

In Section~\ref{sec:steadysection}, we describe a new occurrence of the powered Catalan numbers in terms of lattice paths. 
More precisely, we introduce the family of \emph{steady paths} and prove that they are enumerated by the powered Catalan numbers. 
This is proved by showing a growth for steady paths that is encoded by (a variant called $\Omega_{steady}$ of) the succession rule for permutations avoiding the pattern $\callanpat$. 
We also provide a simple bijection between steady paths and permutations avoiding the vincular pattern $\pat$, 
therefore recovering the enumeration of this family, already known~\cite{larabaxter} to be enumerated by A113227.

Finally, in Section~\ref{sec:bijection} we bridge the gap between the two types of powered Catalan structures, 
by showing a bijection between steady paths (representing the succession rule $\Omega_{steady}$) 
and valley-marked Dyck paths (emblematic of the succession rule $\Omega_{pCat}$). 

\section{Catalan inversion sequences: $\mathbf{I}(\geq,-,\geq)$}\label{sec:catinvseq}

The first family of inversion sequences considered is $\mathbf{I}(\geq,-,\geq)$. It was originally conjectured in~\cite{savage} to be counted by the sequence of Catalan numbers~\cite[A000108]{OEIS} 
(hence the name {\em Catalan inversion sequences}) whose first terms we recall:
\[1,1,2,5,14,42,132, 429, 1430, 4862, 16796, 58786, 208012, 742900,\ldots\]

We note that this conjectured enumeration has recently been proved independently from us by Kim and Lin in~\cite{KL}. 
Their proof does not involve generating trees, but displays a nice Catalan recurrence for 
the filtration $\mathbf{I}_{n,k}(\geq,-,\geq)$ of $\mathbf{I}_{n}(\geq,-,\geq)$ where the additional parameter $k$ is the value of the last element of an inversion sequence. 

We provide another proof of this conjecture in Proposition~\ref{prop:catalaninvseq} by showing that there exists a growth for $\mathbf{I}(\geq,-,\geq)$ 
according to the well-known Catalan succession rule $\Omega_{Cat}$. 
Moreover, we show a second growth for $\mathbf{I}(\geq,-,\geq)$, thereby providing a new Catalan succession rule, 
which is appropriate to be generalized in the next sections. 
In addition, we show a direct bijection between $\mathbf{I}(\geq,-,\geq)$ and a family of pattern-avoiding permutations, which thus results to be enumerated by Catalan numbers. 

\subsection{Combinatorial characterization}

Let us start by observing that the family of Catalan inversion sequences has a simple characterization in terms of inversion sequences avoiding patterns of length three.
\begin{proposition}\label{prop:catwords}
An inversion sequence is in $\mathbf{I}(\geq,-,\geq)$ if and only if it avoids $000$, $100$, $101$, $110$, $201$ and $210$.\end{proposition}
\begin{proof}
The proof is rather straightforward, since containing $e_i$, $e_j$, $e_k$ such that $e_i\geq e_j,e_k$, with $i<j<k$, is equivalent to containing the listed patterns.
\end{proof}
In addition to the above characterization, we introduce the following combinatorial description of Catalan inversion sequences, as it will be useful to define a growth according to the Catalan succession rule $\Omega_{Cat}$.
\begin{proposition}\label{rem:invseq}
Any inversion sequence $e=(e_1,\ldots,e_n)$ is a Catalan inversion sequence if and only if for any $i$, with $1\leq i< n$, 
\begin{center}if $e_i$ forms a \emph{weak descent}, \emph{i.e.} $e_i\geq e_{i+1}$, then $e_i<e_j$, for all $j>i+1$.\end{center}
\end{proposition}
\begin{proof}
The forward direction is clear. The backwards direction can be proved by contrapositive. More precisely, suppose there are three indices $i<j<k$, such that $e_i\geq e_j,e_k$. Then, if $e_j=e_{i+1}$, $e_i$ forms a weak descent and the fact that $e_i\geq e_k$ concludes the proof. Otherwise, 
since $e_i\geq e_j$, there must be an index $i'$, with $i\leq i'<j$, such that $e_{i'}$ forms a weak descent and $e_{i'}\geq e_k$. This concludes the proof as well.
\end{proof}

The previous statement means that any of our inversion sequences has a neat decomposition:  
they are concatenations of shifts of inversion sequences having a single weak descent, at the end. 
A graphical view of this decomposition is shown in Figure~\ref{fig:cat_decomposition}.

\begin{figure}[ht]
\begin{center}
\includegraphics[width=.35\textwidth]{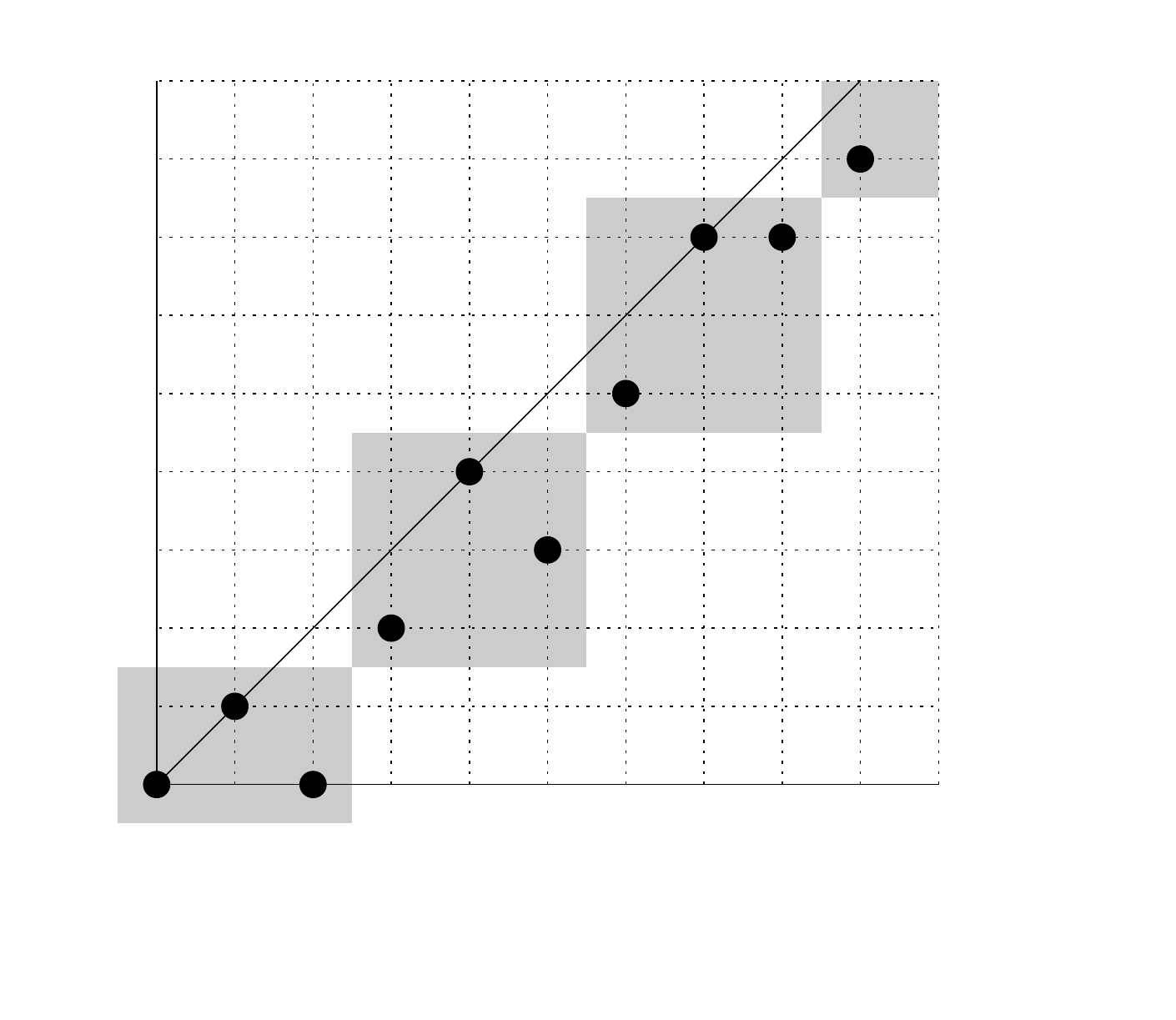}\vspace{-7mm}
\end{center}
\caption{A Catalan inversion sequence and its decomposition.}
\label{fig:cat_decomposition}
\end{figure}

\subsection{Enumerative results}

\begin{proposition}\label{prop:catalaninvseq}
Catalan inversion sequences grow according to  the succession rule $\Omega_{Cat}$,
\[\Omega_{Cat}=\left\{\begin{array}{ll}
(1)\\
\\
(k) &\rightsquigarrow (1),(2), \dots , (k),(k+1). \end{array}\right.\]
\end{proposition}
\begin{proof}
Given an inversion sequence $e=(e_1,\ldots,e_n)$, we define the inversion
sequence $e\odot i$ as the sequence $ (e_1,\ldots,e_{i-1},i-1,e_i,\ldots,e_n)$, where the entry $i-1$ is inserted in position $i$, for some $1\leq i\leq n+1$, and the entries $e_i,\ldots,e_n$ are shifted rightwards by one. By definition of inversion sequences, $i-1$ is the largest possible value that the $i$th entry can assume. 
And moreover, letting $e':=e\odot i$, it holds that  $e'_j=e_{j-1}<j-1$, for all $j>i$; namely the index $i$ is the rightmost index such that $e'_k=k-1$. For example, if $i=4$ and $e=(0,0,1,3,4,5)$, then $e\odot i=(0,0,1,3,3,4,5)$.

Then, we note that given a Catalan inversion sequence $e$ of length $n$, by removing from $e$ the rightmost entry whose value is equal to its position minus one, we obtain a Catalan inversion sequence of length $n-1$. Note that $e_1=0$ for every Catalan inversion sequence, thus such an entry always exists.

Therefore, we can describe a growth for Catalan inversion sequences by inserting an entry $i-1$ in position $i$.
By Proposition~\ref{rem:invseq}, since the entry $i-1$ forms a weak descent in $e\odot i$, the inversion sequence $e\odot i$ is a Catalan inversion sequence of length $n+1$ if and only if $e_{i+1},\ldots,e_n>i-1$. Then, we call \emph{active positions} all the indices $i$, with $1\leq i\leq n+1$, such that $e\odot i$ is  a Catalan inversion sequence of length $n+1$. According to this definition, $n+1$ and $n$ are always active positions: indeed, both $e\odot (n+1)=(e_1,\ldots,e_n,n)$ and $e\odot n=(e_1,\ldots,n-1,e_n)$ are Catalan inversion sequences of length $n+1$.

We label a Catalan inversion sequence $e$ of length $n$ with $(k)$, where the number of active positions is $k+1$. Note that the smallest inversion sequence has label $(1)$, which is the axiom of rule $\Omega_{Cat}$. 

Now, we show that given a Catalan inversion sequence $e$ of length $n$ with label $(k)$, the labels of $e\odot i$, where $i$ ranges over all the active positions, are precisely the label productions of $(k)$ in $\Omega_{Cat}$. 

Let $i_1,\ldots,i_{k+1}$ be the active positions of $e$  from left to right. Note that $i_k=n$ and $i_{k+1}=n+1$. 
We argue that, for any $1\leq j\leq {k+1}$, the active positions of the inversion sequence $e\odot i_j=(e_1,\ldots,i_j-1,e_{i_j},\ldots,e_n)$ are $i_1,\ldots,i_{j-1}$, $n+1$ and $n+2$. 
Indeed, on the one hand any position which is non-active in $e$ is still non-active in $e\odot i_j$. 
On the other hand, by Proposition~\ref{rem:invseq}, the index $i_j$ becomes non-active in $e\odot i_j$, since $e_{i_j}<i_j$ by definition.
Similarly, any position $i_h$, with $i_j<i_h<n+1$, which is active in $e$ becomes non-active in $e\odot i_j$. 
Thus, the active positions of $e\odot i_j$ are $i_1,\ldots,i_{j-1}$, $n+1$ and $n+2$. Hence, $e\odot i_j$ has label $(j)$, for any $1\leq j\leq k+1$.\end{proof}

Furthermore, we can provide a new succession rule for generating Catalan inversion sequences: the growth we provide in the following is remarkable as it allows generalizations in the next sections.
\begin{proposition}\label{prop:catinvseq}
Catalan inversion sequences grow according to the following succession rule
\[\Omega_{Cat_2}=\left\{\begin{array}{lll}
(1,1)\\
\\
(h,k) &\rightsquigarrow &\hspace{-3mm} (0,k+1)^h,\\
&&\hspace{-3mm} (h+1,k), (h+2,k-1), \dots , (h+k,1). \end{array}\right.\]
\end{proposition}
\begin{proof}
We consider the growth of Catalan inversion sequences that consists of adding a
new rightmost entry, and we prove that this growth defines the succession rule $\Omega_{Cat_2}$.
Obviously, this growth is different from the one provided in the proof of Proposition~\ref{prop:catalaninvseq}. 

Let $\max(e)$ be the maximum value among the entries of $e$. And let $\mbox{mwd}(e)$ be the maximum value of the set of all entries $e_i$ that form a weak descent of $e$; if $e$ has no weak descents, then $\mbox{mwd}(e):=-1$. By Proposition~\ref{prop:catwords}, since $e$ avoids $100$, $201$ and $210$, the value $\max(e)$ is $e_{n-1}$ or $e_n$. In particular, if $\max(e)=e_{n-1}\geq e_n$, then $\mbox{mwd}(e)=\max(e)$.

By Proposition~\ref{rem:invseq}, it follows that $f =(e_1,\ldots,e_n,p)$ is a Catalan inversion sequence of length ${n+1}$ if and only if $\mbox{mwd}(e)<p\leq n$. 
Moreover, if $\mbox{mwd}(e)<p\leq\max(e)$, then $e_n$ forms a new weak descent of $f$, and $\mbox{mwd}(f)$ becomes the value $e_n$; whereas, if $\max(e)<p\leq n$, then $\mbox{mwd}(f)=\mbox{mwd}(e)$ since the weak descents of $f$ and $e$ coincide.

Now, we assign to any Catalan inversion sequence  $e$ of length $n$ the label $(h,k)$, where $h=\max(e)-\mbox{mwd}(e)$ and $k=n-\max(e)$.
In other words, $h$ (resp. $k$) marks the number of possible additions smaller than or equal to (resp. greater than) the maximum entry of $e$.

The sequence $e=(0)$ has no weak descents, thus it has label $(1,1)$, which is the axiom of $\Omega_{Cat_2}$.
Let $e$ be a Catalan inversion sequence of length $n$ with label $(h,k)$.
As Figure~\ref{fig:cat_growth} illustrates, the labels of the inversion sequences of length $n+1$ produced by adding a rightmost entry $p$ to $e$ are\begin{itemize}
\item $(0,k+1)$, for any $p\in\{\mbox{mwd}(e)+1,\ldots,\max(e)\}$,
\item $(h + 1,k),(h + 2,k-1),\ldots,(h + k,1)$, when $p=\max(e)+1,\ldots, n$,
\end{itemize}
which concludes the proof that Catalan inversion sequences grow according to $\Omega_{Cat_2}$.
\end{proof}

\begin{figure}[ht]
\begin{center}
\IfFileExists{catalan_growth.pdf}{\includegraphics[width=.6\textwidth]{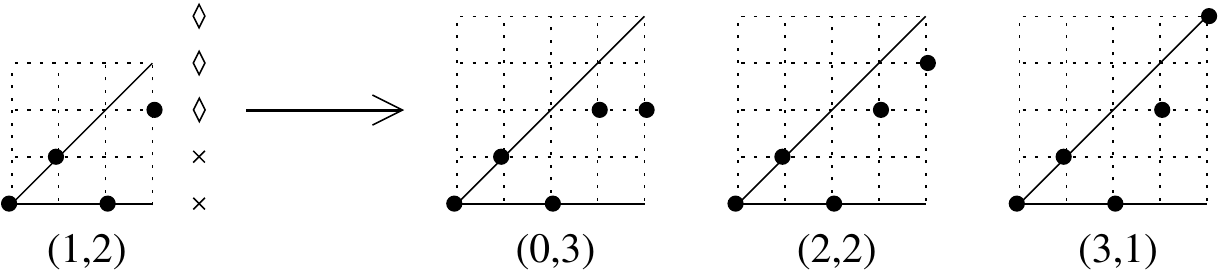}}{MISSING FILE}
\end{center}
\caption{The growth of a Catalan inversion sequence according to $\Omega_{Cat_2}$.}
\label{fig:cat_growth}
\end{figure}

It is well worth noticing that although the above succession rule $\Omega_{Cat_2}$ generates the well-known Catalan numbers, we do not have knowledge of this succession rule in the literature. 

\subsection{One-to-one correspondence with $AV(1\text{-}{23},2\text{-}{14}\text{-}3)$}

In this section we show that Catalan inversion sequences are just {\em left inversion tables} of permutations avoiding the patterns $1\text{-}{23}$ and $2\text{-}{14}\text{-}3$, 
thereby proving that the family of pattern-avoiding permutations $AV(1\text{-}23,2\text{-}14\text{-}3)$ forms a new occurrence of the Catalan numbers. 
We start by recalling some terminology and notation. 

A \emph{(Babson-Steingr\'imsson-)pattern} $\tau$ of length $k$ is any permutation of $\mathcal{S}_k$ where two adjacent entries may or may not be separated by a dash -- see~\cite{babson}. 
Such patterns are also called \emph{generalized} or \emph{vincular}. 
The absence of a dash between two adjacent entries in the pattern indicates that in any pattern-occurrence the two entries are required to be adjacent: 
a permutation $\pi$ of length $n\geq k$ \emph{contains} the vincular pattern $\tau$, if it contains $\tau$ as pattern, 
and moreover, there is an occurrence of the pattern $\tau$ where the entries of $\tau$ not separated by a dash are consecutive entries of the permutation $\pi$; 
otherwise, $\pi$ {\em avoids} the vincular pattern $\tau$. Let $\mathcal{T}$ be a set of patterns. 
We denote by $AV_n(\mathcal{T})$ the family of permutations of length $n$ that avoid any pattern in $\mathcal{T}$, and define $AV(\mathcal{T})=\cup_n AV_n(\mathcal{T})$.

\begin{proposition}
For any $n$, Catalan inversion sequences of length $n$ are in bijection with $AV_n(1\text{-}23,$ $2\text{-}14\text{-}3)$. 
Consequently, the family $AV(1\text{-}23,2\text{-}14\text{-}3)$ is enumerated by Catalan numbers.
\end{proposition}
\begin{proof}
The second part of the statement is a immediate consequence of the first part, which we now prove.

Let $\mathtt{T}$ be the mapping associating to each $\pi\in\mathcal{S}_n$ its left inversion table $\mathtt{T}(\pi)=(t_1,\ldots,t_n)$. 
We will use many times the following simple fact: for every $i<j$, if $\pi_i>\pi_j$ (\emph{i.e.}  the pair $(\pi_i,\pi_j)$ is an {\em inversion}), then $t_i>t_j$.

Let $\mathtt{R}$ be the reverse operation on arrays. 
We can prove our statement by using the mapping $\mathtt{R}\circ\mathtt{T}$, which is a bijection between the family $\mathcal{S}_n$ of permutations and integer sequences $(e_1,\ldots,e_n)$ such that $0\leq e_i<i$.
We will simply show that the restriction of the bijection $\mathtt{R}\circ\mathtt{T}$ to the family $AV(1\text{-}23,2\text{-}14\text{-}3)$ yields a bijection with Catalan inversion sequences. Precisely, we want to prove that for every $n$, an inversion sequence is in the set $\{(\mathtt{R}\circ\mathtt{T})\,(\pi): \pi\in AV_n(1\text{-}23,2\text{-}14\text{-}3)\}$ if and only if it is a Catalan inversion sequence of length $n$ (\emph{i.e.} belongs to $\mathbf{I}_n(\geq,-,\geq)$).
\begin{itemize}
\item[$\Rightarrow)$] We prove the contrapositive: if $e\not\in \mathbf{I}_n(\geq,-,\geq)$, then $\pi=(\mathtt{R}\circ\mathtt{T})^{-1}(e)$ contains $1\text{-}23$ or $2\text{-}14\text{-}3$. Let $t=(t_1,\ldots,t_n)=(e_n,\ldots, e_1)$. Then, $t$ is the left inversion table of a permutation $\pi\in\mathcal{S}_n$, \emph{i.e.} $\mathtt{T}(\pi)=t$.
Since $e\not\in \mathbf{I}_n(\geq,-,\geq)$, there exist three indices, $i<j<k$, such that $t_i\leq t_k$ and $t_j\leq t_k$. 

Without loss of generality, we can suppose that there is no index $h$, such that $j<h<k$ and $t_i\leq t_h$ and $t_j\leq t_h$. Namely $t_k$ is the leftmost entry of $t$ that is at least as large as both $t_i$ and $t_j$. 
Then, we have two possibilities:
\begin{enumerate}
\item either $j+1=k$,
\item or $j+1\neq k$, and in this case it holds that $t_j>t_{k-1}$ or $t_i>t_{k-1}$.\end{enumerate}
First, from $t_i\leq t_k$ and $t_j\leq t_k$ it follows that $\pi_i<\pi_k$ and $\pi_j<\pi_k$. 

Now, we prove that both in case 1. and in case 2. above we have $\pi\not\in AV_n(1\text{-}23,2\text{-}14\text{-}3)$.
\begin{enumerate}
\item Let us consider the subsequence $\pi_i\pi_j\pi_{j+1}$. We have $\pi_i<\pi_{j+1}$ and $\pi_j<\pi_{j+1}$. If also $\pi_i<\pi_j$, then it forms a $1\text{-}23$. 

Otherwise, it must hold that $\pi_i>\pi_j$, and thus $t_j<t_i\leq t_{j+1}$. Since the pair $(\pi_i,\pi_j)$ is an inversion of $\pi$ and $t_i\leq t_{j+1}$, there must be a point $\pi_s$ on the right of $\pi_{j+1}$ such that $(\pi_{j+1},\pi_s)$ is an inversion and $(\pi_i,\pi_s)$ is not. Thus, $\pi_i\pi_j\pi_{j+1}\pi_s$ forms a $2\text{-}14\text{-}3$.

\item  First, if $t_j>t_{k-1}$, consider the subsequence $\pi_i\pi_j\pi_{k-1}\pi_k$. It follows that $t_{k-1}<t_k$, since $t_j\leq t_k$, and thus $\pi_{k-1}<\pi_k$. In addition, we know that $\pi_j<\pi_k$.
Then, $\pi_j\pi_{k-1}\pi_k$ forms an occurrence of  $1\text{-}23$ if $\pi_j<\pi_{k-1}$. Otherwise, it must hold that $\pi_j>\pi_{k-1}$. As in case 1.,  the pair $(\pi_j,\pi_{k-1})$ is an inversion, and $t_j\leq t_k$. Therefore, there must be an element $\pi_s$ on the right of $\pi_k$ such that $(\pi_k,\pi_s)$ is an inversion and $(\pi_j,\pi_s)$ is not. Hence $\pi_j\pi_{k-1}\pi_{k}\pi_s$ forms a $2\text{-}14\text{-}3$.

Now, suppose $t_j\leq t_{k-1}$, and consider the subsequence $\pi_i\pi_j\pi_{k-1}\pi_k$. According to case 2., it must be that $t_i>t_{k-1}$, and since $t_i\leq t_k$, it holds that $t_{k-1}<t_k$. Since both $\pi_j<\pi_{k-1}$ and $\pi_{k-1}<\pi_k$ hold, $\pi_j\pi_{k-1}\pi_k$ forms an occurrence of $1\text{-}23$. 
\end{enumerate}
\item[$\Leftarrow)$] By contrapositive, if a permutation $\pi$ contains $1\text{-}23$ or $2\text{-}14\text{-}3$, then $e=(\mathtt{R}\circ\mathtt{T})(\pi)$ is not in $\mathbf{I}(\geq,-,\geq)$.
\begin{itemize}
\item[-] If $\pi$ contains $1\text{-}23$, there must be two indices $i$ and $j$, with $i<j$, such that $\pi_i\pi_j\pi_{j+1}$ forms an occurrence of $1\text{-}23$. We can assume that no points $\pi_{i'}$ between $\pi_i$ and $\pi_j$ are such that $\pi_{i'}<\pi_i$. Otherwise we consider $\pi_{i'}\pi_j\pi_{j+1}$ as our occurrence of $1\text{-}23$.

Then, two relations hold: $t_i\leq t_{j+1}$ and $t_j\leq t_{j+1}$, and thus $e\not\in\mathbf{I}(\geq,-,\geq)$.

\item[-] If $\pi$ contains $2\text{-}14\text{-}3$, and avoids $1\text{-}23$, there must be three indices $i,j$ and $k$, with $i<j<j+1<k$, such that $\pi_i\pi_j\pi_{j+1}\pi_{k}$ forms an occurrence of $2\text{-}14\text{-}3$. We can assume that no points $\pi_{i'}$ between $\pi_i$ and $\pi_j$ are such that $\pi_{i'}<\pi_i$. Indeed, in case $\pi_{i'}<\pi_j$ held, $\pi_{i'}\pi_j\pi_{j+1}$ would be an occurrence of  $1\text{-}23$; whereas, if $\pi_j<\pi_{i'}<\pi_i$, we could consider $\pi_{i'}\pi_j\pi_{j+1}\pi_k$ as our occurrence of  $2\text{-}14\text{-}3$. 

Then, as above $t_j\leq t_{j+1}$, and $t_i+1\leq t_{j+1}$ because $(\pi_i,\pi_j)$ is an inversion of $\pi$. Nevertheless, $(\pi_{j+1},\pi_k)$ is an inversion of $\pi$ as well, and $\pi_i<\pi_k$. Thus, $t_i\leq t_{j+1}$ and $e\not\in\mathbf{I}(\geq,-,\geq)$. \qedhere
\end{itemize}
\end{itemize}
\end{proof}
We mention that although inversion sequences are actually a coding for permutations, 
it is often not easy (if at all possible) to characterize the families $\mathbf{I}(\rho_1,\rho_2,\rho_3)$ in terms of families of pattern-avoiding permutations. 
A few examples of bijective correspondences between pattern-avoiding inversion sequences and pattern-avoiding permutations have been collected in~\cite{savage}. 
We report below the examples of~\cite{savage} where the permutations are defined by the avoidance of \emph{classical} patterns: 
\begin{itemize}
 \item $\mathbf{I}(=,-,-)$ and $AV(123,132,231)$, \cite[Theorem 1]{savage};
 \item $\mathbf{I}(<,\neq,-)$ and $AV(213,321)$, \cite[Theorem 9]{savage};
  \item $\mathbf{I}(=,<,-)$ and $AV(132,231)$, \cite[Section 2.6.1]{savage};
  \item $\mathbf{I}(<,\geq,-)$ and $AV(213,312)$, \cite[Theorem 16]{savage};
\item $\mathbf{I}(-,>,-)$ and $AV(213)$, \cite[Theorem 27]{savage};
 \item $\mathbf{I}(>,<,-)$ and $AV(2143,3142,4132)$ and $AV(2143,3142,3241)$, \cite[Theorem 37-38]{savage};
\item $\mathbf{I}(>,-,\geq)$ and $AV(2134,2143)$, \cite[Theorem 40]{savage};
\item $\mathbf{I}(\geq,\neq,\geq)$ and $AV(4321,4312)$, \cite[Theorem 45]{savage}.
\end{itemize}

In addition, \cite[Theorem 56]{savage} shows a bijective correspondence
between $\mathbf{I}(>,\neq,>)$ and a family of permutations avoiding a specific marked mesh pattern. 
Our case of $\mathbf{I}(\geq,-,\geq)$ and $AV(1\text{-}{23},2\text{-}{14}\text{-}3)$ 
shows another example of such bijective correspondences, where the excluded patterns on permutations are however \emph{vincular}. 

\section{Inversion sequences $\mathbf{I}(\geq,\geq,\geq)$}\label{sec:invseq3cross}

Following the hierarchy of Figure~\ref{fig:hierarchy}, the next family we turn to is $\mathbf{I}(\geq,\geq,\geq)$.
This family was originally conjectured in~\cite{savage} to be counted by sequence A108307 on~\cite{OEIS}, which is defined as the enumerative sequence of set partitions of $\{1,\ldots, n\}$ that avoid enhanced 3-crossings~\cite{MBM_Xin}.
In~\cite[Proposition 2]{MBM_Xin} it is proved that the
number $E_3(n)$ of these set partitions is given by $E_3(0)=E_3(1)=1$ and the recursive relation
\begin{equation}\label{eq:rec3cros}
8(n+3)(n+1)E_3(n)+(7n^2+53n+88)E_3(n+1)-(n+8)(n+7)E_3(n+2)=0\,,
\end{equation}
which holds for all $n\geq0$.
Thus, the first terms of sequence A108307 according to recurrence~\eqref{eq:rec3cros} are
\[1, 1, 2, 5, 15, 51, 191, 772, 3320, 15032, 71084, 348889, 1768483, 9220655, 49286863,\ldots\]

At the conference \emph{Permutation Patterns 2017} in Reykjavik, 
we presented~\cite{SimonePP} a proof that the enumerative sequence of the family $\mathbf{I}(\geq,\geq,\geq)$ is indeed the sequence A108307. 
Our proof works as follows. 
First, we build a generating tree for $\mathbf{I}(\geq,\geq,\geq)$, 
which is encoded by a succession rule that generalizes the one in Proposition~\ref{prop:catinvseq}.
Then, we solve the resulting functional equation using
a variant of the so-called kernel method -- see~\cite{BM,KM1} and references therein -- 
which is sometimes referred to as \emph{obstinate} kernel method. 
The Lagrange inversion formula can then be applied to yield a closed formula for 
the number of inversion sequences in $\mathbf{I}_n(\geq,\geq,\geq)$.
And finally, using the method of creative telescoping, 
we deduce from this closed formula a recurrence satisfied by the considered enumerating sequence.

The details of this proof are not provided in the following. 
The interested reader may however find them in a previous version of our paper~\cite{arxivV1}, 
or in the PhD thesis of the third author~\cite[Section 5.2]{TesiVeronica}. 
The reason for this omission is that essentially the same proof has been independently found by Lin~\cite{LinKernel}. 
In the following, we simply give some statements that constitute the main steps of the proof, 
together with a reference to the corresponding statements in the paper of Lin. 

We also point out to the interested reader that Yan~\cite{Yan} has now also provided a bijective proof 
that inversion sequences in $\mathbf{I}_n(\geq,\geq,\geq)$ and set partitions avoiding enhanced 3-crossings 
are enumerated by the same sequence.

\subsection{Combinatorial characterization}
To start, we provide a combinatorial description of the family $\mathbf{I}(\geq,\geq,\geq)$, 
which is useful to prove Proposition~\ref{prop:3crosinvseq}.

As Figure~\ref{fig:hierarchy} shows, the
family $\mathbf{I}(\geq,\geq,\geq)$ properly includes $\mathbf{I}(\geq,-,\geq)$ as a subfamily. 
For instance, the inversion sequence $(0,0,1,1,4,2,6,5)$ is both in $\mathbf{I}_8(\geq,-,\geq)$ 
and in $\mathbf{I}_8(\geq,\geq,\geq)$, while $(0,1,0,1,4,2,3,5)$ is not in $\mathbf{I}_8(\geq,-,\geq)$ 
despite being in $\mathbf{I}_8(\geq,\geq,\geq)$. The following characterization makes this fact explicit.

\begin{proposition}\label{prop:3crosscharacterisation}
An inversion sequence belongs to $\mathbf{I}(\geq,\geq,\geq)$ if and only if it avoids $000$, $100$, $110$ and $210$.
\end{proposition}

\begin{proof}
The proof is a quick check that containing  $e_i,e_j,e_k$ such that $e_i\geq e_j\geq e_k$, with $i<j<k$, is equivalent to containing the above patterns.
\end{proof}

The above result makes clear that every Catalan inversion sequence is in $\mathbf{I}(\geq,\geq,\geq)$.
In addition, Proposition~\ref{prop:3crosscharacterisation} proves the following property stated in~\cite[Observation 7]{savage}.

\begin{remark}
\label{rem:3crosinvseq}
Let any inversion sequence $e=(e_1,\ldots,e_n)$ be decomposed into two subsequences $e^{LTR}$, 
which is the increasing sequence of left-to-right maxima of $e$ (\emph{i.e.} entries $e_i$ such that $e_i>e_j$, for all $j<i$), 
and $e^{bottom}$, which is the (possibly empty) sequence comprised of all the remaining entries of $e$.

Then, an inversion sequence $e$ is in the set $\mathbf{I}(\geq,\geq,\geq)$ if and only if 
$e^{LTR}$ and $e^{bottom}$ are both strictly increasing sequences -- 
see decomposition in Figure~\ref{fig:I2_char} where the sequence $e^{LTR}$ is highlighted. 
\end{remark}

\begin{figure}[ht]
\begin{center}
\IfFileExists{inter1.pdf}{\includegraphics[width=.35\textwidth]{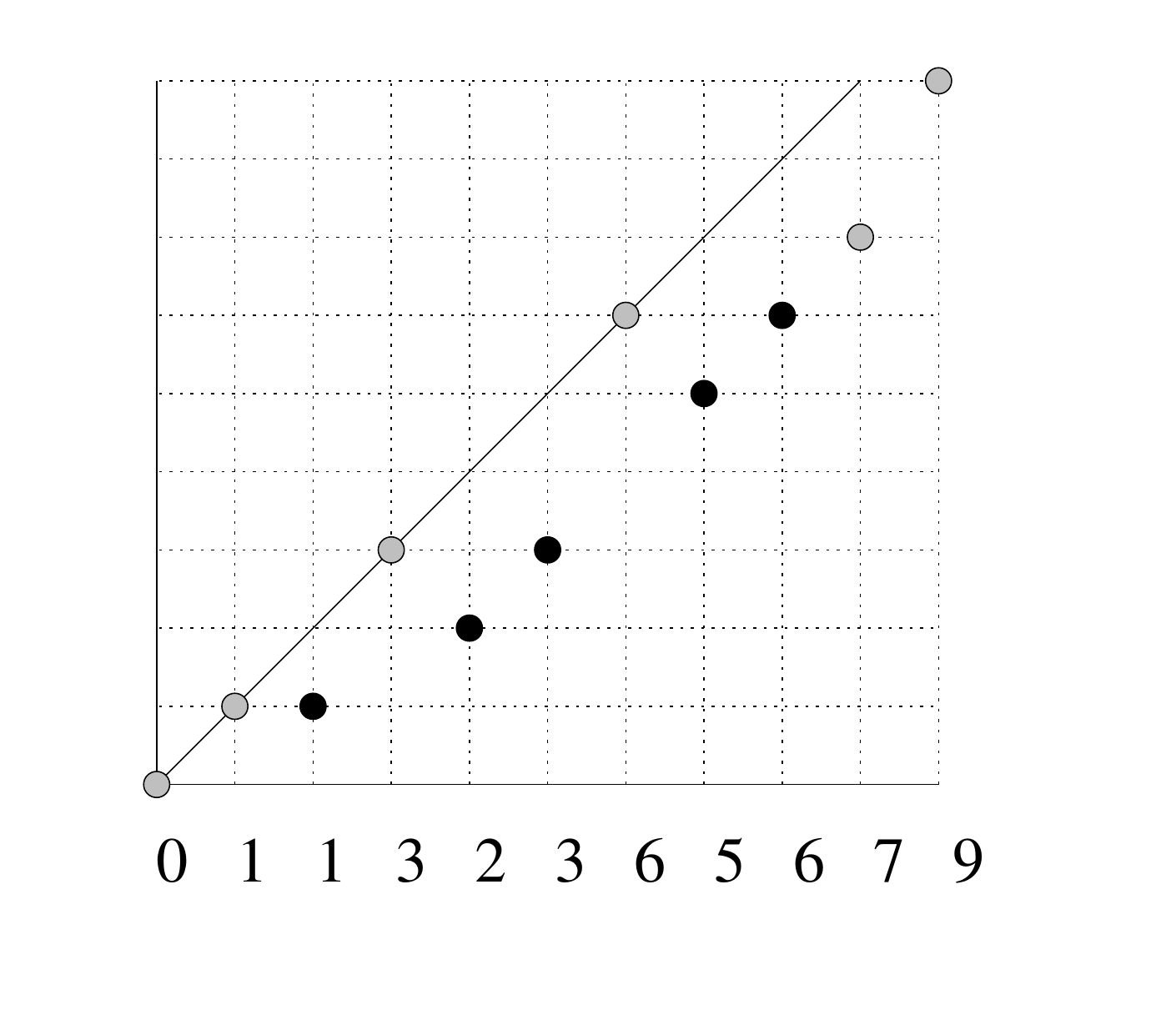}\vspace{-7mm}}{MISSING FILE}
\end{center}
\caption{An inversion sequence in $\mathbf{I}(\geq,\geq,\geq)$ and its decomposition according to $e^{LTR}$ and $e^{bottom}$.}
\label{fig:I2_char}
\end{figure}

\subsection{Enumerative results}
\begin{proposition}\label{prop:3crosinvseq}
The family $\mathbf{I}(\geq,\geq,\geq)$ grows according to the following succession rule
\[\Omega_{\mathbf{I}(\geq,\geq,\geq)}=\left\{\begin{array}{lll}
(1,1)\\
\\
(h,k) &\rightsquigarrow &\hspace{-3mm} (h-1,k+1),(h-2,k+1),\ldots,(0,k+1),\\
&&\hspace{-3mm} (h+1,k),(h+2,k-1), \dots , (h+k,1). \end{array}\right.\]
\end{proposition}

This proposition corresponds to Lemma 2.2 in~\cite{LinKernel}. 
It is proved by letting inversion sequences of $\mathbf{I}(\geq,\geq,\geq)$ grow by adding a new rightmost entry, 
and by giving to each such inversion sequence $e$ a label as follows.  
Let $\max(e)$ is the maximum value of $e$ and $\mbox{last}(e)$ be the rightmost entry of $e^{bottom}$, if there is any, otherwise $\mbox{last}(e):=-1$. 
The label $(h,k)$ of $e$ is then defined by $h=\max(e)-\mbox{last}(e)$ and $k=n-\max(e)$. 
The growth of inversion sequences of $\mathbf{I}(\geq,\geq,\geq)$ is illustrated in Figure~\ref{fig:lin}. 

\begin{figure}[ht]
\begin{center}
\includegraphics[width=.75\textwidth]{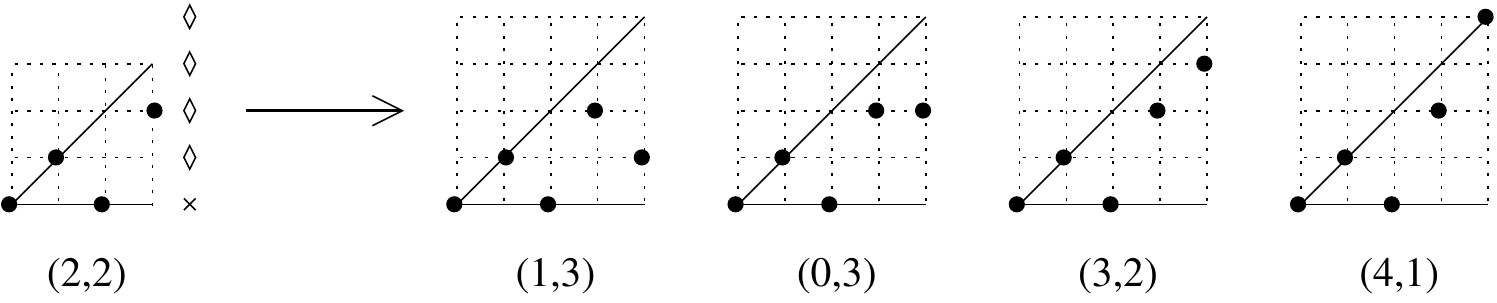}\vspace{-7mm}
\end{center}
\caption{The growth of inversion sequences of $\mathbf{I}(\geq,\geq,\geq)$.}
\label{fig:lin}
\end{figure}

The next steps toward the enumeration of the family $\mathbf{I}(\geq,\geq,\geq)$ are to translate 
the succession rule of Proposition~\ref{prop:3crosinvseq} into a functional equation, 
and then to solve it. 

For $h,k\geq0$, let $A_{h,k}(x)\equiv A_{h,k}$ denote the size generating function of inversion sequences of the family $\mathbf{I}(\geq,\geq,\geq)$ having label $(h,k)$. 
The rule $\Omega_{\mathbf{I}(\geq,\geq,\geq)}$ translates using a standard technique into a functional equation for the generating function $A(x;y,z)\equiv A(y,z)=\sum_{h,k\geq0}A_{h,k}\,y^hz^k$.

\begin{proposition}
The generating function $A(y,z)$ satisfies the following functional equation
\begin{equation}\label{eq:3crosfe}
A(y,z)=xyz+\frac{xz}{1-y}\left(A(1,z)-A(y,z)\right)+\frac{xyz}{z-y}\left(A(y,z)-A(y,y)\right)\,.
\end{equation}
\end{proposition}

The above statement coincides with Proposition 2.3 in~\cite{LinKernel}. 

Equation~\eqref{eq:3crosfe} is a linear functional equation with two catalytic variables, $y$ and $z$, in the sense of Zeilberger~\cite{ZeilbergerCAT}.
Similar functional equations have been solved by using the {\em obstinate kernel method} (see~\cite{BM,semibaxterlong}, and references therein), 
which allows us to provide the following expression for the generating function of $\mathbf{I}(\geq,\geq,\geq)$.
Note that the same method was also applied in~\cite{LinKernel} to derive the following theorem (Theorem 3.1 in~\cite{LinKernel}).
\begin{thm}\label{thm:GF3cros}
Let $W(x;a)\equiv W$ be the unique formal power series in $x$ such that

\begin{equation*}
W=x\bar{a}(W+1+a)(W+a+a^2).
\end{equation*}

The series solution $A(y,z)$ of Equation~\eqref{eq:3crosfe} satisfies 
\[A(1+a,1+a)=\left[\frac{Q(a,W)}{(1+a)^3}\right]^\geq,\] 
where $Q(a,W)$ is a polynomial in $W$ whose coefficients are Laurent polynomials in $a$ defined by
\begin{align*}
Q(a,W)&=\left(-\frac{1}{a^6}-\frac{3}{a^5}-\frac{3}{a^4}-\frac{1}{a^3}+1+3a+3a^2+a^3\right)\,W\\
&+\left(\frac{1}{a^5}+\frac{1}{a^4}-\frac{1}{a}-1\right)\,W^2
+\left(\frac{1}{a^6}-\frac{1}{a^4}+\frac{1}{a^3}-\frac{1}{a}\right)\,W^3
+\left(-\frac{1}{a^5}+\frac{1}{a^4}\right)\,W^4,
\end{align*}
and the notation $[Q(a,W)/(1+a)^3]^\geq$ stands for the formal power series in $x$ obtained by considering 
only those terms in the series expansion of $Q(a,W)/(1+a)^3$ that have non-negative powers of $a$.
\end{thm}

Note that $W$ and $Q(a,W)$ are algebraic series in $x$ whose coefficients are Laurent polynomials in $a$. 
It follows, as in~\cite[page 6]{BM}, that $A(1+a,1+a)$ is D-finite. 
Hence, the specialization $A(1,1)$, which is the generating function of $\mathbf{I}(\geq,\geq,\geq)$, is D-finite as well. 

Applying the Lagrange inversion formula to the expression of $A(1+a,1+a)$ in Theorem~\ref{thm:GF3cros}, 
we can derive an explicit, yet very complicated, expression for the coefficients of the generating function $A(1,1)$.
Although this expression is complicated, 
Zeilberger's method of creative telescoping~\cite{zeilberger,cretele} can be applied to it, 
and provides a much simpler recursive formula satisfied by these numbers. 
This is also how the proof that $\mathbf{I}(\geq,\geq,\geq)$ is enumerated by~\cite[A108307]{OEIS} is concluded in~\cite{LinKernel}, 
giving the following statement. 

\begin{proposition}\label{prop:3cros2}
Let $a_n=|\mathbf{I}_n(\geq,\geq,\geq)|$.
The numbers $a_n$ are recursively defined by $a_0=a_1=1$ and for $n\geq0$,
\[8(n+3)(n+1)a_n+(7n^2+53n+88)a_{n+1}-(n+8)(n+7)a_{n+2}=0\,.\]

Thus, $\{a_n\}_{n\geq0}$ is sequence A108307 on~\cite{OEIS}.
\end{proposition}

\section{Baxter inversion sequences: $\mathbf{I}(\geq,\geq,>)$}\label{sec:invseq3rd}
The next family of inversion sequences according to the hierarchy of Figure~\ref{fig:hierarchy} is $\mathbf{I}(\geq,\geq,>)$. 
This family of inversion sequences was originally conjectured in~\cite{savage} to be counted by the sequence A001181 \cite{OEIS} of Baxter numbers, whose first terms are
\[1,2,6,22,92,422,2074,10754,58202,326240, 1882960, 11140560,67329992, \ldots\]
This conjecture has recently been proved in~\cite[Theorem 4.1]{KL}. 
Accordingly, we call $\mathbf{I}(\geq,\geq,>)$ the family of {\em Baxter inversion sequences}. 

The proof of \cite[Theorem 4.1]{KL} is analytic. 
Precisely, \cite[Lemma 4.3]{KL} provides a succession rule for $\mathbf{I}(\geq,\geq,>)$. 
It is then shown to generate Baxter numbers using the obstinate kernel method and the results in~\cite[Section 2]{BM}. 
This succession rule is however not a classical one associated with Baxter numbers, and 
no other Baxter family is known to grow according to this new Baxter succession rule. 
It would be desirable to establish a closer link (either via generating trees, or via bijections) 
between $\mathbf{I}(\geq,\geq,>)$ and any other known Baxter family. 

\subsection{Combinatorial characterization}
The family of Baxter inversion sequences clearly contains $\mathbf{I}(\geq,\geq,\geq)$, as shown by the following characterization. 
\begin{proposition}\label{prop:baxterinv}
An inversion sequence is a Baxter inversion sequence if and only if it avoids $100$,  $110$ and $210$.
\end{proposition}
\begin{proof}
The statement is readily checked, as in Propositions~\ref{prop:catwords} and~\ref{prop:3crosscharacterisation}.
\end{proof}
Another characterization for this family is the following. 
Recall that for an inversion sequence $e=(e_1,\ldots,e_n)$, we call an entry $e_i$ a LTR maximum (resp. RTL minimum), if $e_i>e_j$, for all $j<i$ (resp. $e_i<e_j$, for all $j>i$).
\begin{proposition}\label{rem:baxinvseq} 
An inversion sequence $e=(e_1,\ldots,e_n)$ is a Baxter inversion sequence if and only if 
for every $i$ and $j$, with $i<j$ and $e_i>e_j$, both $e_i$ is a LTR maximum and $e_j$ is a RTL minimum. 
\end{proposition}
\begin{proof}
The proof in both directions is straightforward by considering the characterization of Proposition~\ref{prop:baxterinv}.
\end{proof}

\subsection{Enumerative results}
We choose to report here a proof of~\cite[Lemma 4.3]{KL} (which is omitted in~\cite{KL}). 
This proof is essential in our work, since it displays 
a growth for Baxter inversion sequences that generalizes the one for the family $\mathbf{I}(\geq,\geq,\geq)$ provided in Proposition~\ref{prop:3crosinvseq}.

\begin{proposition}\label{prop:baxinvseq}
Baxter inversion sequences  grow according to the following succession rule
\[\Omega_{Bax}=\left\{\begin{array}{lll}
(1,1)\\
\\
(h,k) &\rightsquigarrow &\hspace{-3mm} (h-1,k+1),\ldots,(1,k+1),\\
&&\hspace{-3mm} (1,k+1),\\
&&\hspace{-3mm} (h+1,k), \dots , (h+k,1). \end{array}\right.\]
\end{proposition}
\begin{proof}
We show that the growth of Baxter inversion sequences by addition of a new rightmost entry 
(as in the proofs of Propositions~\ref{prop:catinvseq} and~\ref{prop:3crosinvseq}) 
can be encoded by the above succession rule $\Omega_{Bax}$.

As in the proof of Proposition~\ref{prop:3crosinvseq}, let $\mbox{last}(e)$ be the value of the rightmost entry of $e$ which is not a LTR maximum, if there is any. Note that $\mbox{last}(e)$ is also the largest value that is not a LTR maximum, since $e$ avoids $210$ by Proposition~\ref{prop:baxterinv}. Otherwise, if such an entry does not exist, we set $\mbox{last}(e)$ equal to the smallest value of $e$, \emph{i.e.} $\mbox{last}(e):=0$. 

Moreover, if this rightmost entry of $e$ which is not a LTR maximum exists, it can either form an inversion (\emph{i.e.} there exists an entry $e_i$ on its left such that $e_i>\mbox{last}(e)$) or not. 
We need to distinguish two cases in order to define the addition of a new rightmost entry to $e$: 
\begin{itemize}
\item[\textbf{(a)}] in case either all the entries of $e$ are LTR maxima, or the rightmost entry of $e$ which is not a LTR maximum does not form an inversion;
\item[\textbf{(b)}] in case the rightmost entry of $e$ which is not a LTR maximum exists and does form an inversion.
\end{itemize}
Then, according to Proposition~\ref{rem:baxinvseq}, we have that
\begin{itemize}
\item[\textbf{(a)}] The sequence $f=(e_1,\ldots,e_n,p)$ is  a Baxter inversion sequence of length ${n+1}$ if and only if $\mbox{last}(e)\leq p\leq n$. Moreover, if $\mbox{last}(e)\leq p<\max(e)$, where as usual $\max(e)$ is the maximum value of $e$, then $\mbox{last}(f)=p$ and $f$ falls in case {\bf (b)}. Else if $p=\max(e)$, then again $\mbox{last}(f)=p$, yet $f$ falls in case {\bf (a)}. While, if $\max(e)<p\leq n$, $p$ is a LTR maximum of $f$, which thus falls in the same case {\bf (a)} as $e$, and $\mbox{last}(f)=\mbox{last}(e)$.

\item[\textbf{(b)}] The sequence $f=(e_1,\ldots,e_n,p)$ is  a Baxter inversion sequence of length ${n+1}$ if and only if $\mbox{last}(e)<p\leq n$. In particular, if $\mbox{last}(e)<p<\max(e)$, then $\mbox{last}(f)=p$ and $f$ falls in case {\bf (b)}. Else if $p=\max(e)$, then again $\mbox{last}(f)=p$ and $f$ falls in case {\bf (a)}. While, if $\max(e)<p\leq n$, as above $p$ is a LTR maximum of $f$, which thus falls in the same case {\bf (b)} as $e$, and $\mbox{last}(f)=\mbox{last}(e)$.
\end{itemize}

Now, we assign to any Baxter inversion sequence $e$ of length $n$ a label according to the above distinction:
in case {\bf (a)} (resp. {\bf (b)}) we assign the label $(h,k)$, where $h=\max(e)-\mbox{last}(e)+1$ (resp. $h=\max(e)-\mbox{last}(e)$) and $k=n-\max(e)$.

The sequence $e=(0)$ of size one falls in case {\bf (a)}, thus it has label $(1,1)$, which is the axiom of $\Omega_{Bax}$.
Now, let $e$ be a Baxter inversion sequence of length $n$ with label $(h,k)$.
Following the above distinction, the inversion sequences of length $n+1$ produced by adding a rightmost entry $p$ to $e$ have labels:
\begin{itemize}
\item[\textbf{(a)}] $ $
\vspace{-7mm}
\begin{itemize}
\item[$\bullet$] $(h-1,k+1),\ldots,(1,k+1)$, when $p=\mbox{last}(e),\ldots,\max(e)-1$,
\item[$\bullet$] $(1,k+1)$, for $p=\max(e)$,
\item[$\bullet$] $(h + 1,k),(h + 2,k-1),\ldots,(h + k,1)$, when $p=\max(e)+1,\ldots, n$,
\end{itemize}
\item[\textbf{(b)}] $ $
\vspace{-7mm}
\begin{itemize}
\item[$\bullet$] $(h-1,k+1),\ldots,(1,k+1)$, when $p=\mbox{last}(e)+1,\ldots,\max(e)-1$,
\item[$\bullet$] $(1,k+1)$, for $p=\max(e)$,
\item[$\bullet$] $(h + 1,k),(h + 2, k-1),\ldots,(h + k,1)$, when $p=\max(e)+1,\ldots,n$,
\end{itemize}
\end{itemize}
which concludes the proof that Baxter inversion sequences  grow according to $\Omega_{Bax}$.
\end{proof}
The growth of Baxter inversion sequences is depicted in Figure~\ref{fig:bax_growth}.

\begin{figure}[ht]
\begin{center}
\IfFileExists{baxter_growth.pdf}{\includegraphics[width=.75\textwidth]{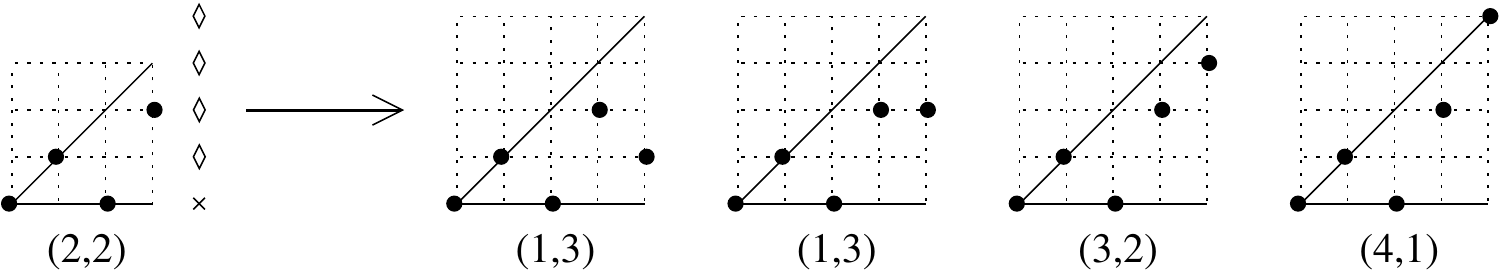}}{MISSING FILE}
\end{center}
\caption{The growth of Baxter inversion sequences.}
\label{fig:bax_growth}
\end{figure}

\section{Semi-Baxter inversion sequences: $\mathbf{I}(\geq,>,-)$}\label{sec:invseq4th}
The hierarchy of Figure~\ref{fig:hierarchy} continues with the family $\mathbf{I}(\geq,>,-)$.
This family of inversion sequences was originally conjectured in~\cite{savage} to be counted by the sequence A117106~\cite{OEIS}. 
The validity of this conjecture follows from~\cite{semibaxterlong} where the numbers of sequence A117106 are named semi-Baxter, hence the name {\em semi-Baxter inversion sequences}. 
We recall the first terms of this enumeration sequence
\[1,2,6,23,104,530,2958,17734,112657, 750726, 5207910, 37387881, 276467208, \ldots\] 

\subsection{Combinatorial characterization}
\begin{proposition}
An inversion sequence is in $\mathbf{I}(\geq,>,-)$ if and only if it avoids $110$ and $210$.
\end{proposition}

The proof of the above statement is elementary, and we omit it. Yet it shows clearly that semi-Baxter inversion sequences avoid only two of the three patterns avoided by the Baxter inversion sequences (see Proposition~\ref{prop:baxterinv} for a comparison). 

Moreover, the following characterization is an extension of that provided in Proposition~\ref{rem:baxinvseq} for the family $\mathbf{I}(\geq,\geq,>)$.
Recall that for $e=(e_1,\ldots,e_n)$ an inversion sequence, 
we call an entry $e_i$ a LTR maximum if $e_i>e_j$, for all $j<i$, and we say that $e_i$ and $e_j$ form an inversion if $i<j$ and $e_i>e_j$.

\begin{proposition}\label{rem:sbaxinvseq} 
An inversion sequence $e=(e_1,\ldots,e_n)$ is in $\mathbf{I}(\geq,>,-)$ if and only if 
for every $e_i$ and $e_j$ that form an inversion, $e_i$ is a LTR maximum. 
\end{proposition}
\begin{proof}
Similarly to the proof of Proposition~\ref{rem:baxinvseq}, the above statement follows immediately by considering that $e$ is an inversion sequence of $\mathbf{I}(\geq,>,-)$ if and only if it avoids $110$ and $210$.
\end{proof}

\subsection{Enumerative results}

For the sake of completeness, we choose to report here a direct proof of the fact that the family $\mathbf{I}(\geq,>,-)$ can be generated by the rule $\Omega_{semi}$ associated with semi-Baxter numbers. 
It allows us to see that the growth of the family $\mathbf{I}(\geq,\geq,>)$ in the proof of  Proposition~\ref{prop:baxinvseq} can be easily generalized to a growth for the family $\mathbf{I}(\geq,>,-)$. 

As proved in~\cite[Section 3]{semibaxterlong}, once Proposition~\ref{prop:baxinvseq} has been established, 
the enumeration of the family $\mathbf{I}(\geq,>,-)$ is obtained by applying the obstinate kernel method as discussed for the family $\mathbf{I}(\geq,\geq,\geq)$. 

\begin{proposition}\label{prop:finalinvseq}
The family $\mathbf{I}(\geq,>,-)$ grows according to the following succession rule
\[\Omega_{semi}=\left\{\begin{array}{lll}
(1,1)\\
\\
(h,k) &\rightsquigarrow &\hspace{-3mm} (h,k+1),\ldots,(1,k+1),\\
&&\hspace{-3mm} (h+1,k), \dots , (h+k,1). \end{array}\right.\]
\end{proposition}
\begin{proof}
As previously, we define a growth for the family $\mathbf{I}(\geq,>,-)$ according to $\Omega_{semi}$ by adding a new rightmost entry. 

As in the proof of Proposition~\ref{prop:baxinvseq}, let $\mbox{last}(e)$ be the value of the rightmost entry of $e$ which is not a LTR maximum, if there is any. Otherwise, $\mbox{last}(e):=0$. 
Note that differently from Proposition~\ref{prop:baxinvseq}, here we do not need to distinguish cases depending on whether or not the rightmost entry of $e$ not being a LTR maximum forms an inversion. 

According to Proposition~\ref{rem:sbaxinvseq},
it follows that $f =(e_1,\ldots,e_n,p)$ is an inversion sequence of $\mathbf{I}_{n+1}(\geq,>,-)$ if and only if $\mbox{last}(e)\leq p\leq n$. Moreover, if $\mbox{last}(e)\leq p\leq\max(e)$, where as usual $\max(e)$ is the maximum value of $e$, then $\mbox{last}(f)=p$; if $\max(e)<p\leq n$, then $\mbox{last}(f)=\mbox{last}(e)$, since $p$ is a LTR maximum.

Now, we assign to any $e\in\mathbf{I}_n(\geq,>,-)$ the label $(h,k)$, where $h=\max(e)-\mbox{last}(e)+1$ and $k=n-\max(e)$.

The sequence $e=(0)$ of size one has label $(1,1)$, which is the axiom of $\Omega_{semi}$, since $\mbox{last}(e)=0$.
Let $e$ be an inversion sequence of $\mathbf{I}_n(\geq,>,-)$ with label $(h,k)$.
The labels of the inversion sequences of $\mathbf{I}_{n+1}(\geq,>,-)$ produced adding a rightmost entry $p$ to $e$ are
\begin{itemize}
\item $(h,k+1),\ldots,(1,k+1)$, when $p=\mbox{last}(e),\ldots,\max(e)$,
\item $(h + 1,k),(h + 2, k-1),\ldots,(h + k,1)$, when $p=\max(e)+1,\ldots, n$,
\end{itemize}
which concludes the proof that $\mathbf{I}(\geq,>,-)$ grows according to $\Omega_{semi}$.
\end{proof}
The growth semi-Baxter inversion sequences is depicted in Figure~\ref{fig:semibax_growth}.

\begin{figure}[ht]
\begin{center}
\IfFileExists{semi_baxter_growth.pdf}{\includegraphics[width=.9\textwidth]{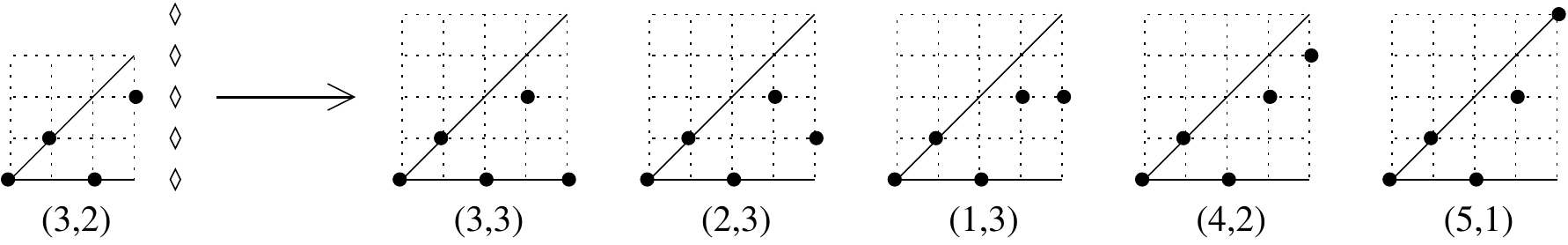}}{MISSING FILE}
\end{center}
\caption{The growth of semi-Baxter inversion sequences.}
\label{fig:semibax_growth}
\end{figure}

\section{Powered Catalan inversion sequences: $\mathbf{I}(=,>,>)$}\label{sec:callaninvseq}

The family of inversion sequences $\mathbf{I}(=,>,>)$ is the last element of the chain in Figure~\ref{fig:hierarchy}. 
These objects that can be characterized by the avoidance of $110$ are completely enumerated in~\cite[Theorem 13]{corteel}, and are called {\em powered Catalan inversion sequences} in the following.
The associated enumerative number sequence (of \emph{powered Catalan numbers}) is registered on~\cite{OEIS} as sequence A113227, which also counts permutations avoiding $\callanpat$. 
Its first terms are
\[1, 1, 2, 6, 23, 105, 549, 3207, 20577, 143239, 1071704, 8555388, 72442465,647479819,\ldots\]

Enumerating permutations avoiding $\callanpat$ 
(or rather increasing ordered trees with increasing leaves, with whom they are in bijective correspondence), 
Callan~\cite{callan} proved that the $n$th term of the sequence A113227 can be obtained as $p_n=\sum_{k=0}^nc_{n,k}$, where $c_{n,k}$ is recursively defined by

\begin{equation}\label{eq:reccallan}
\begin{cases}
c_{0,0}=1,&\\
c_{n,0}=0,&\text{for }n\geq1\\
c_{n,k}=c_{n-1,k-1}+k\sum_{j=k}^{n-1} c_{n-1,j},&\text{for }n\geq1,\mbox{ and }1\leq k\leq n\,.
\end{cases}
\end{equation}
%

\medskip

The proof that $\mathbf{I}(=,>,>)$ is enumerated by sequence A113227 of \cite{OEIS} uses the above recurrence, which characterizes this sequence. 
Namely, Corteel \emph{et al} proved the following. 

\begin{proposition}[\cite{corteel}, Theorem 13]\label{prop:110enum}
For $n\geq1$ and $0\leq k\leq n$, the number of powered Catalan inversion sequences having $k$ zeros is given by the term $c_{n,k}$ of Equation~\eqref{eq:reccallan}. 
Thus, the number of powered Catalan inversion sequences of length $n$ is $p_n$, for every $n\geq1$.
\end{proposition}

Proposition~\ref{prop:110enum} can be rephrased in terms of succession rules, as done below with the rule $\Omega_{pCat}$.
More precisely, for $n\geq1$ and $k\geq1$, the number of nodes at level $n$ that carry the label $(k)$ in the generating tree associated with $\Omega_{pCat}$ 
is precisely the quantity $c_{n,k}$ given by Equation~\eqref{eq:reccallan}. 

\begin{proposition}\label{prop:pocatrule}
The family of powered Catalan inversion sequences grows according to the following succession rule
\[\Omega_{pCat}=\left\{\begin{array}{lll}
(1)\\
\\
(k) &\rightsquigarrow &\hspace{-3mm} (1),(2)^2,(3)^3,\ldots,(k)^k,(k+1). \end{array}\right.\]
\end{proposition}

We notice that $\Omega_{pCat}$ is extremely similar to the Catalan succession rule $\Omega_{Cat}$ (see page~\pageref{page:OmegaCat}): specifically, the productions of $\Omega_{pCat}$ are the same as in $\Omega_{Cat}$, but with multiplicities appearing as ``powers''. Hence, the name \emph{powered Catalan}. 

\begin{proof}[Proof of Proposition~\ref{prop:pocatrule}]
We prove the above statement by showing a growth for the family of powered Catalan inversion sequences. 
Let $e=(e_1,\ldots,e_n)$ be a powered Catalan inversion sequence of length $n$. Suppose that $e$ has $k$ entries equal to $0$, and let $i_1,\ldots,i_k$ be their indices. Since $e$ avoids $110$, there could be at most one single entry equal to $1$ among $e_1,\ldots,e_{i_k}$.
We define a growth that changes the number of $0$'s and $1$'s entries of $e$, as follows.
\begin{itemize}
\item[a)] First, increase by one all the entries of $e$ that are greater than $0$; namely $e'=(e_1',\ldots,e_n')$, where $e_i'=e_i$, if $i=i_1,\ldots,i_k$, otherwise $e_i'=e_i+1$. Note that $e_1'=e_1=0$, and that $e'$ does not have to be an inversion sequence.
\item[b)] Then, insert a new leftmost $0$ entry; namely $e''=(0,e_1',\ldots,e_n')$. Note that $e''$ is an inversion sequence of size $n+1$, and moreover it has $k+1$ entries equal to $0$ at positions $1,i_1+1,\ldots,i_k+1$ and no entries equal to $1$.
\item[c)] Build the following inversion sequences, starting from $e''$.
\begin{itemize}[leftmargin=1cm]
\item[\textbf{(1)}] Replace all the zeros at positions $i_1+1,i_2+1,\ldots,i_k+1$ by $1$; namely $e^{(1)}=(0,e^*_1,\ldots,e^*_n)$, where $e^*_i=e_i+1$, for all $i$. Note that $e^{(1)}$ has only one entry equal to $0$, and exactly $k$ entries equal to $1$.
\item[\textbf{(j)}] For all $1<j<k+1$, replace all the zeros at positions $i_{j+1}+1,\ldots,i_k+1$ by $1$, and furthermore replace by $1$ only one zero entry among those at indices $i_1+1,\ldots,i_j+1$. 
There are thus $j$ different inversion sequences $e^{(m)}=(0,e^*_1,\ldots,e^*_n)$, with $1\leq m\leq j$, such that $e^*_i=e_i+1$ except for the indices $i_1+1,\ldots,i_{m-1}+1,i_{m+1}+1,\ldots,i_j+1$. 
Note that $e^{(m)}$ has exactly $j$ entries equal to $0$, for any $1\leq m\leq j$, and $k+1-j$ entries equal to $1$.
\item[\textbf{(k+1)}] Set $e^{(k+1)}=e''$.
\end{itemize}
\end{itemize}
Note that all the inversion sequences of size $n+1$ produced at step c) avoid $110$, since the initial inversion sequence $e$ avoids $110$, 
and the modifications performed in steps a) - b) - c) may result in at most one $1$ to the left of the rightmost $0$. 
Thus, in each of the above cases we build a powered Catalan inversion sequence of length $n+1$.

Moreover, given any powered Catalan inversion sequence $f$ of length $n+1$, it is easy to retrieve the unique inversion sequence $e$ of length $n$ that produces $f$ according to the operations of c).
Indeed, it is sufficient to replace all the entries equal to $1$ by $0$, remove the leftmost $0$ entry, and finally decrease by one all the entries greater than $0$. (This procedure is a) - b) - c) backwards.) 
In addition, the operation of c) which generates $f$ is uniquely determined by the number and relative positions of the $0$ and $1$ entries of $f$. 

Finally, we label a powered Catalan inversion sequence $e$ of length $n$ with $(k)$, where $k$ is its number of $0$ entries. It is straightforward, and the above itemized list suggests it, that the powered Catalan inversion sequences produced by $e$ following the construction at step c) have  labels $(1),(2)^2,(3)^3,\ldots,(k)^k,(k+1)$.
\end{proof}

The sequence of powered Catalan numbers proves to be extremely rich with combinatorial interpretations, 
and quite a few enumerative problems associated with it are open, or beg for a more natural proof. 
We collect some of them in the remainder of this article, 
and solve a few by providing bijections between families enumerated by the powered Catalan numbers. 

\section{Powered Catalan numbers}\label{sec:callan}

Recall from the previous section that the sequence of powered Catalan number $(p_n)$ (A113227 on~\cite{OEIS}) 
is defined by $p_n=\sum_{k=0}^nc_{n,k}$, where the term $c_{n,k}$ is recursively defined by Equation~\eqref{eq:reccallan}. 
To our knowledge, there is no information about the ordinary generating function $F_{pCat}(x)=\sum_{n\geq0}p_nx^n$. 
On the contrary, the exponential generating function $E_{pCat}(x)=\sum_{n\geq0}p_nx^n/n!$ has been studied in~\cite{1234sergi}, as well as in~\cite{callan}, 
where by means of the recurrence~\eqref{eq:reccallan} a refined version of this exponential generating function is provided.

\subsection{Known combinatorial structures enumerated by the powered Catalan numbers}\label{sec:mvpaths}

\begin{definition}
A \emph{valley-marked Dyck path} (see Figure~\ref{fig:VMdyck}) of semi-length $n$ is a Dyck path $P$ of length $2n$ in which, for each valley (\emph{i.e.} $DU$ factor), one of the lattice points between the valley vertex and the $x$-axis is marked. 
In other words, if $(i,k)$ pinpoints any valley of $P$, then a valley-marked Dyck path associated with $P$ must have a mark in a point $(i,j)$, where $0\leq j\leq k$. 
If $j=0$, we say that the valley has a \emph{trivial mark}.
\end{definition}

In addition, we say that a \emph{return to the mark} of a valley-marked Dyck path is any valley whose mark is on the valley itself. 
Note that returns to the $x$-axis are a special case of return to the mark. 
(Here and everywhere, when speaking of return to the $x$-axis, we do not include the starting nor the ending point of the path.) 

We also define the \emph{total mark} of a valley-marked Dyck path as the sum of the heights of the marks. 

\begin{figure}[ht!]
\centering
\IfFileExists{VMdyck.pdf}{\includegraphics[width=0.6\textwidth]{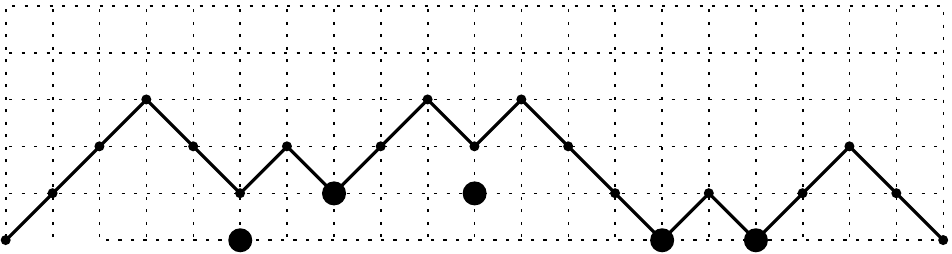}}{MISSING FILE}
\caption{A valley-marked Dyck path.}
\label{fig:VMdyck}
\end{figure}

Valley-marked Dyck paths are enumerated by powered Catalan number according to their semi-length. More precisely, 
the number of valley-marked Dyck paths of semi-length $n$ having $k$ down steps in the last descent 
(or symmetrically, $k$ up steps on the main diagonal (of equation $x=y$)) is given by 
the term $c_{n,k}$ of Equation~\eqref{eq:reccallan}, for every $n\geq0$ and $0\leq k\leq n$ (\cite[Section 7]{callan}). 

\paragraph{Increasing ordered trees}
Another family of objects counted by sequence A113227 is one of labeled ordered trees~\cite{callan}.
\begin{definition}
An \emph{increasing ordered tree} of size $n$ is a plane tree with $n+1$ labeled vertices, the standard label set being $\{0,1,2,\ldots,n\}$, such that each
child exceeds its parent.
An increasing ordered tree {\em has increasing leaves} if its leaves, taken in pre-order, are increasing.
\end{definition}
Figure~\ref{fig:IOT} shows two increasing ordered trees, the first has increasing leaves, while the second does not.
\begin{figure}[ht!]
\centering
\IfFileExists{IOT.pdf}{\includegraphics[width=0.45\textwidth]{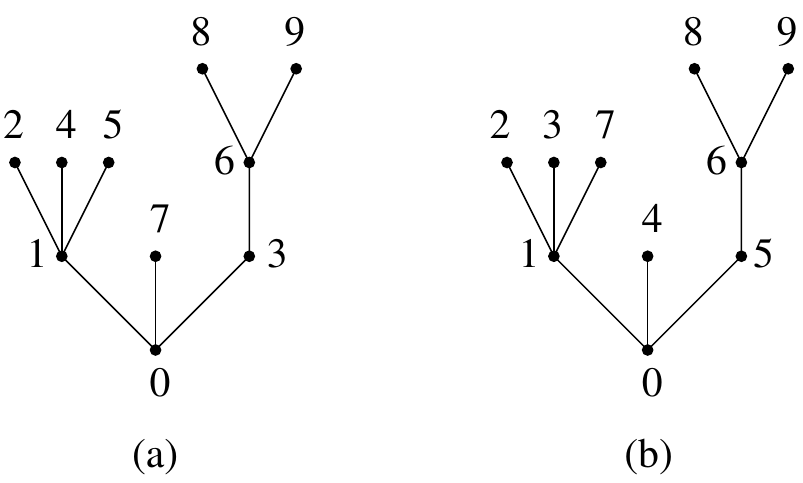}\vspace{-4mm}}{MISSING FILE}
\caption{Two increasing ordered trees: (a) with increasing leaves; (b) with non-increasing leaves.}
\label{fig:IOT}
\end{figure}

The number of increasing ordered trees of size $n$ is given by the odd double factorial $(2n-1)!!$~\cite{callan2}. If we require the additional constraint of having increasing leaves, then the number of these increasing ordered trees of size $n$ results to be the $n$th powered Catalan number~\cite{callan}. More precisely, the number of increasing ordered trees with increasing leaves of size $n$ and root degree $k$ is given by the number $c_{n,k}$ of Equation~\eqref{eq:reccallan}.

\begin{remark}
It is not hard to prove (and details are left to the reader) that valley-marked Dyck paths and increasing ordered trees with increasing leaves have a growth according to $\Omega_{pCat}$.

Specifically, we assign to any valley-marked Dyck path $P$ a label $(k)$, 
where $k$ is the number of down steps in the last descent. 
A growth according to $\Omega_{pCat}$ is thus defined by adding a new rightmost peak (\emph{i.e.} UD factor) in any point of $P$'s last descent, 
and (if a new rightmost valley is generated) by marking a lattice point between the new rightmost valley and the $x$-axis. 
Turning to increasing ordered trees with increasing leaves, 
the label of such a tree is $(k)$, where $k$ is the root degree. 
A growth for these trees according to $\Omega_{pCat}$ is defined as follows: 
first increase each vertex label $\ell>0$ by one; 
then select any (possibly empty) bunch of contiguous root edges; 
and finally insert a new vertex with label $1$ that is child of the root and parent of the selected bunch of edges.
\end{remark}

\paragraph{Pattern-avoiding permutations}
Some families of pattern-avoiding permutations are known to be counted by the sequence of powered Catalan numbers.
Indeed, the sequence A113227 is actually registered on~\cite{OEIS} as the enumerative sequence of permutations avoiding the generalized pattern $\callanpat$. 

The family $AV(\callanpat)$ has been enumerated by Callan in~\cite{callan}, providing a bijection between $AV(\callanpat)$ and increasing ordered trees with increasing leaves. 

\smallskip

Furthermore, in~\cite{larabaxter}, some other families of pattern-avoiding permutations are presented as related to the sequence A113227~\cite{OEIS}. 
In particular, in~\cite{larabaxter} and subsequent papers, the families  $AV(1\text{-}32\text{-}4)$, and  $AV(\pat)$, and $AV(1\text{-}43\text{-}2)$ are proved to be equinumerous to permutations avoiding $\callanpat$.
It has been conjectured in~\cite{shattuck} that also the family $AV(23\text{-}1\text{-}4)$ is equinumerous to $AV(\callanpat)$. 
We attempted to prove this conjecture by defining a growth for the family $AV(23\text{-}1\text{-}4)$ according to $\Omega_{pCat}$. 
Although our attempts were not successful, they lead us to refine this conjecture as follows.

\begin{con}\label{con:shattuck}
The number of permutations of $AV_n(23\text{-}1\text{-}4)$ with $k$ RTL minima is given by $c_{n,k}$ as defined in Equation~\eqref{eq:reccallan}.
\end{con}
Although we have a little evidence of the above fact (only up to $n=9$), we suspect that a growth for permutations avoiding the pattern $2\text{-}1\text{-}3$ according to $\Omega_{Cat}$, where the label $(k)$ marks the number of RTL minima, could be generalized as to obtain one for permutations avoiding $23\text{-}1\text{-}4$ according to $\Omega_{pCat}$. We leave open the problem of finding such a growth.

\subsection{A second succession rule for powered Catalan numbers}\label{sec:2ndrulepcat}

The bijection provided by Callan between increasing ordered trees with increasing leaves and $AV(\callanpat)$ in~\cite{callan} is quite intricate. 
And the interpretation of the parameter $k$ in $c_{n,k}$ is rather complicated on permutations avoiding $\callanpat$. 
This suggests that the combinatorics of $\callanpat$-avoiding permutations is quite different from that of other powered Catalan objects previously presented in our paper. 
This is also supported by the fact that we can describe a natural growth for $AV(\callanpat)$ 
which is not encoded by  $\Omega_{pCat}$. 
This leads us to present a second succession rule associated with powered Catalan numbers (denoted $\Omega_{\callanpat}$ below). 
Our impression is that a powered Catalan family is naturally generated by either $\Omega_{pCat}$ or $\Omega_{\callanpat}$, but not by both. 
This is further discussed at the beginning of Section~\ref{sec:bijection}. 

\medskip

Following Callan~\cite{callan}, we observe that permutations avoiding $\callanpat$ have a simple characterization in terms of LTR minima and RTL maxima, as follows.

\begin{proposition}\label{1-23-4}
A permutation $\pi$ of length $n$ belongs to $AV(\callanpat)$ if and only if for every index $1\leq i < n$, 
\begin{center}if $\pi_i \pi_{i+1}$  is an ascent (\emph{i.e.}, $\pi_i<\pi_{i+1}$), then $\pi_i$ is a LTR minimum {\bf or}  $\pi_{i+1}$ is a RTL maximum. \end{center}
\end{proposition}

\begin{proof}
Suppose that there exists an index $i$, $1\leq i<n$, such that $\pi_i<\pi_{i+1}$, and neither $\pi_i$ is a LTR minimum nor  $\pi_{i+1}$ is a RTL maximum. Then, there exists an index $j<i$ such that $\pi_j<\pi_i$, and an index $k>i+1$ such that $\pi_k>\pi_{i+1}$. Thus, $\pi_j\pi_i\pi_{i+1}\pi_k$ forms an occurrence of $\callanpat$. Conversely, if $\pi$ contains an occurrence of $\callanpat$, by definition of pattern containment there exists an index $i$, $1\leq i<n$, such that $\pi_i<\pi_{i+1}$, and neither $\pi_i$ is a LTR minimum nor  $\pi_{i+1}$ is a RTL maximum.
\end{proof}

We show now a recursive growth for the family $AV(\callanpat)$ that yields a succession rule for powered Catalan numbers whose labels are arrays of length two. 

\begin{proposition}\label{prop:1-23-4}
Permutations avoiding $\callanpat$ grows according to the following succession rule
\[\Omega_{\callanpat}=\left\{\begin{array}{lllr}
(1,1)&&&\\
\\
(1,k)&\rightsquigarrow&\hspace{-3mm} (1,k+1), (2,k), \ldots , (1+k,1), \\[1ex]
(h,k)&\rightsquigarrow&\hspace{-3mm}(1,h+k), (2,h+k-1), \ldots,  (h, k+1),\\
&&\hspace{-3mm}(h+1,0),\ldots,(h+k,0), &\qquad\mbox{if }h\neq1.
 \end{array}  \right.
\]
\end{proposition}
\begin{proof}
First, observe that removing the rightmost point of a permutation avoiding $\callanpat$, we obtain a permutation that still avoids $\callanpat$. 
So, a growth for the permutations avoiding $\callanpat$ can be obtained with local expansions on the right. 
We denote by $\pi\cdot a$, where $a \in \{1, \dots, n+1\}$, the permutation $\pi' = \pi'_1 \dots \pi'_{n}\pi'_{n+1}$ where $\pi'_{n+1}=a$, and $\pi'_i = \pi_i$, if $\pi_i < a$, $\pi'_i = \pi_i +1$ otherwise. 

For $\pi$ a permutation in $AV_n(\callanpat)$, 
the active sites of $\pi$ are by definition 
the points $a$ (or equivalently the values $a$) such that $\pi \cdot a$ avoids $\callanpat$. The other points $a$ are called non-active sites. 

An occurrence of $1\text{-}23$ in $\pi$ is a subsequence $\pi_j \pi_i \pi_{i+1}$ (with $j<i$) such that $\pi_j<\pi_i<\pi_{i+1}$. 
Note that the non-active sites $a$ of $\pi$ are the values larger than $\pi_{i+1}$, for some occurrence $\pi_j \pi_i \pi_{i+1}$ of $1\text{-}23$. 
Then, given $\pi\in AV_n(\callanpat)$, we denote by $\pi_s\pi_{t-1}\pi_t$ the occurrence of $1\text{-}23$ (if there is any), in which the point $\pi_t$ is minimal. Then the active sites of $\pi$ form a consecutive sequence from the bottommost site to $\pi_t$, \emph{i.e.} they are $[1,\pi_t]$. Figure~\ref{fig:1234growth} should help understanding which sites are active (represented by diamonds, as usual). If $\pi\in AV_n(\callanpat)$ has no occurrence of $1\text{-}23$, then the active sites of $\pi$ are $[1,n+1]$.

Now, we assign a label $(h,k)$ to each permutation $\pi\in AV_n(\callanpat)$, 
where $h$ (resp. $k$) is the number of its active sites  smaller than or equal to (resp. greater than) $\pi_n$.
Remark that $h\geq1$, since $1$ is always an active site. Moreover, $h=\pi_n$: indeed, let $\pi_s\pi_{t-1}\pi_t$ be the occurrence of $1\text{-}23$ with $\pi_t$ minimal. It must hold that $\pi_t\geq\pi_n$, otherwise $\pi_s\pi_{t-1}\pi_t\pi_n$ would form an occurrence of $\callanpat$.

The label of the permutation $\pi=1$ is $(1,1)$, which is the axiom in $\Omega_{\callanpat}$. 
The proof then is concluded by showing that for any $\pi\in AV_n(\callanpat)$ of label $(h,k)$, the permutations $\pi \cdot a$ have labels according to the productions of $\Omega_{\callanpat}$ when $a$ runs over all active sites of $\pi$.
To prove this we need to distinguish whether $\pi_n=1$ or not. 

If $\pi_n=1$, no new occurrence of $1\text{-}23$ can be generated in the permutation $\pi\cdot a$, for any $a$ active site of $\pi$. Thus, the active sites of $\pi\cdot a$ are as many as those of $\pi$ plus one (since the active site $a$ of $\pi$ splits into two actives sites of $\pi \cdot a$). Then, since $\pi_n=1$, permutation $\pi$ has label $(1,k)$, for some $k>0$ (at least one site above $1$ in active), and permutations $\pi\cdot a$, for $a$ ranging over all the active sites of $\pi$ from bottom to top, have labels $(1,k+1), (2,k),\ldots,(1+k,1)$, which is the first production of $\Omega_{\callanpat}$.

Otherwise, we have that $\pi$ has label $(h,k)$, with $h>1$, and $\pi_n=h$. In this case a new occurrence of $1\text{-}23$ is generated in the permutation $\pi\cdot a$, for every $a>\pi_n$: indeed, $1\,\pi_na$ forms an occurrence of $1\text{-}23$, and moreover is such that $a$ is minimal. Else if $a\leq \pi_n$, no new occurrence of $1\text{-}23$ can be generated in the permutation $\pi\cdot a$. Thus, permutations $\pi\cdot a$ have labels $(1,h+k), (2,h+k-1),\ldots,(h,k+1)$, for any active site $a\leq\pi_n$, and labels $(h+1,0),(h+2,0),\ldots,(h+k,0)$, for any active site $a>\pi_n$. Note that this label production coincides with the two lines of the second production of $\Omega_{\callanpat}$ concluding the proof.
Figure~\ref{fig:1234growth} shows an example of the above construction.
\end{proof}
\begin{figure}[ht!]
\centering
\IfFileExists{1234growth.pdf}{\includegraphics[width=\textwidth]{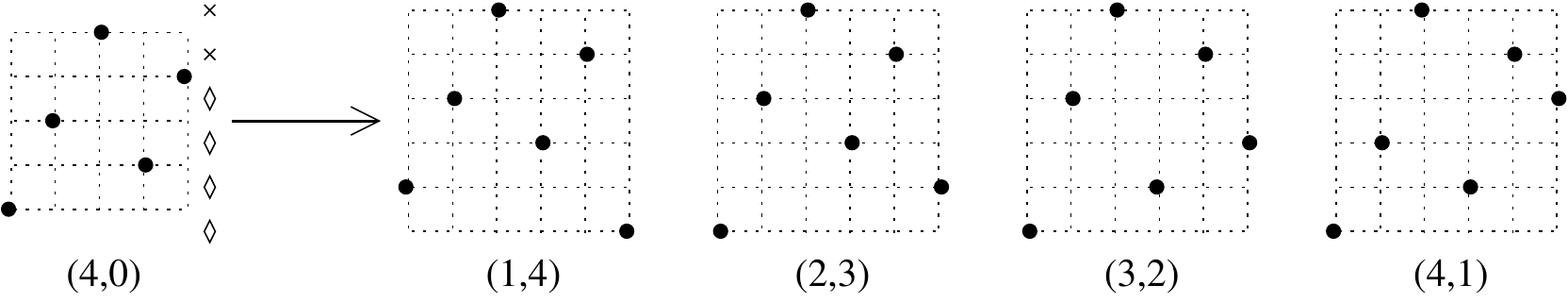}\vspace{-7mm}}{MISSING FILE}
\caption{The growth of a permutation avoiding $\callanpat$ of label $(4,0)$.}
\label{fig:1234growth}
\end{figure}

\section{The family of steady paths}\label{sec:steadysection}

In this last part of the paper we provide a further (and new) combinatorial interpretation of powered Catalan numbers in terms of lattice paths. 


%


\subsection{Definition and succession rule}\label{sec:steadydef}

\begin{definition}\label{def:steady}
We call a \emph{steady path} of size $n$ a lattice path $P$ confined to the cone  $\mathfrak{C}=\{(x,y)\in\mathbb{N}^2: y\leq x\}$, which uses $U=(1,1)$, $D=(1,-1)$ and $W=(-1,1)$ as steps, without any factor $WD$ nor
$DW$, starting at $(0,0)$ and ending at $(2n,0)$, such that:
\begin{description}
\item{(S1)}  for any factor $UU$, the suffix of $P$ following this $UU$ factor lies weakly below the line parallel to $y=x$ passing through the $UU$ factor;
\item{(S2)}  for any factor $WU$, the suffix of $P$ following this $WU$ factor lies weakly below the line parallel to $y=x$ passing through the up step of the $WU$ factor.
\end{description}
We call the \emph{edge line} of $P$ the line $y=x-t$, with $t\geq0$ an even integer, which supports the up step of the rightmost occurrence of either $UU$ or $WU$ in $P$. 
\end{definition}

The name ``steady'' is motivated by the two restrictions (S1) and (S2), which force these paths to remain weakly below a line that moves rightwards and  conveys more stability to the mountain range the path would represent. 
Figure~\ref{fig:prud}~(a) shows an example of a steady path whose edge line coincides with $y=x$, whereas the edge line of the steady path depicted in Figure~\ref{fig:prud}~(b) is $y=x-6$. Figures~\ref{fig:prud}~(c),(d) show two different examples of paths confined to $\mathfrak{C}$ that are not steady paths. The reader may observe that steady paths can be regarded as a subfamily of those ``skew Dyck paths'' considered in \cite{simoncino} and enumerated according to several parameters. 

\begin{figure}[ht]
\begin{center}
\IfFileExists{prud.pdf}{\includegraphics[width=\textwidth]{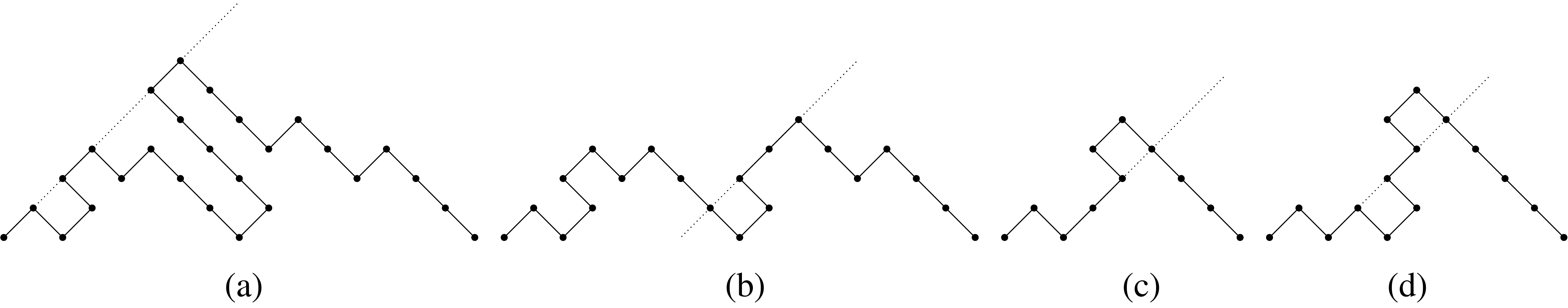}\vspace{-7mm}}{MISSING FILE}
\end{center}
\caption{(a) An example of a steady path $P$ of size $8$ with edge line $y=x$; (b) An example of a steady path $P$ of size $8$ with edge line $y=x-6$; (c) a path in $\mathfrak{C}$ that violates (S1); (d) a path in $\mathfrak{C}$ that violates (S2).}
\label{fig:prud}
\end{figure}
\begin{remark}\label{rem:steady}
By Definition~\ref{def:steady}, the size of a steady path $P$ is equal to the number of its $U$ steps.
Moreover, any steady path $P$ of size $n$ is uniquely determined by the set of positions of its up steps $U^{(1)},\ldots,U^{(n)}$ recorded from left to right: precisely, by the set of starting points $(i_k,j_k)$ for any $U^{(k)}$. Indeed, since neither $WD$ nor $DW$ can occur, there is only one way to draw a steady path given the set of positions $\{(0,0)=(i_1,j_1),\ldots,(i_n,j_n)\}$ of its up steps from left to right. 

Furthermore, let a set of points $\{(i_1,j_1),\ldots,(i_n,j_n)\}$ in $\mathfrak{C}$ be such that for every index $1\leq k\leq n$, $j_k=-i_k+2(k-1)$. This set uniquely defines a steady path of size $n$ provided that for every $1< k\leq n$, if $i_k\leq i_{k-1}+1$, then all the points $(i_\ell,j_\ell)$, with $\ell>k$, lie weakly below the line parallel to $y=x$ passing through the point $(i_k,j_k)$.
\end{remark}

We provide a growth for the family of steady paths that results in the following proposition.
\begin{proposition}\label{prop:steady_growth}
The family of steady paths grows according to the following succession rule
\[\Omega_{steady}=
\left\{\begin{array}{lll}
(0,2)&&\\
\\
(h,k) &\rightsquigarrow \hspace{-3mm}& (h+k-1,2), \ldots , (h+1,k),\\
& \hspace{-3mm}&(0,k+1), \ldots , (0,h+k+1). 
\end{array}\right.\]
\end{proposition}

\begin{proof}
By Remark~\ref{rem:steady}, given a steady path $P$ of size $n$, we obtain a steady path of size $n-1$ if we remove its rightmost point $(i_n,j_n)$, namely the rightmost up step of $P$. 
This allows us to provide a growth for steady paths by addition of a new rightmost up step. 

Let $P$ be a steady path of size $n$, and $(0,0)=(i_1,j_1),\ldots,(i_n,j_n)$ be the positions of its up steps. 
We describe in which position $(i_{n+1},j_{n+1})$ a new rightmost up step can be inserted so that the path obtained is still a steady path.
Specifically, according to Definition~\ref{def:steady}, if the edge line of $P$ is $y=x-2t$, with $t$ a non-negative integer, 
then the point $(i_{n+1},j_{n+1})$ must remain weakly below this line, that is $j_{n+1}\leq i_{n+1}-2t$. 
So, we add a new rightmost up step in any position $(2n,0),(2n-1,1),(2n-2,2),\ldots,(2n-s,s)$, where $s=n-t$. 
By Remark~\ref{rem:steady}, there exists a unique path of size $n+1$ corresponding to $(0,0)=(i_1,j_1),\ldots,(i_n,j_n),(i_{n+1},j_{n+1})$, 
where $(i_{n+1},j_{n+1})$ is any point among $(2n,0),(2n-1,1),(2n-2,2),\ldots,(2n-s,s)$, and it is steady by construction.

Moreover, the positions $(2n,0),(2n-1,1),(2n-2,2),\ldots,(2n-s,s)$ can be divided into two groups: the positions that are ending points of $D$ steps of the last descent of $P$, and those which are not. This distinction is crucial. Indeed, when we insert a $U$ step in an ending point of a $D$ step of $P$'s last descent, no factors $WU$ or $UU$ are generated. On the contrary, denoting $(2n-r,r)$ the topmost point of the last descent of $P$, when we insert the new rightmost $U$ step at position $(2n-r,r)$, a $UU$ factor is formed, and when we insert it in any point $(2n-i,i)$, with $r<i\leq s$, a $WU$ factor is formed. In both cases, the edge line of the obtained steady path must pass through the point $(2n-r,r)$ (resp. $(2n-i,i)$, for $r< i\leq s$). 
Thus, the edge line may move rightwards so as to include this point.

Now, we assign the label $(h,k)\equiv(h,r+1)$ to any steady path $P$ of size $n$ and edge line $y=x-2t$, where $r\geq1$ is the number of steps in the last descent of $P$ and $h=(n-t)-r$. In other words, the label interpretation is such that $h$ counts the positions in which we insert a new rightmost $U$ step that do not belong to the last descent of $P$.

The steady path $UD$ of size $1$ has edge line $y=x$. Thus its label is $(0,2)$, which is the axiom of $\Omega_{steady}$.
Given a steady path $P$ of size $n$, edge line $y=x-2t$, and label $(h,k)\equiv(h,r+1)$, we now prove that the labels of the steady paths obtained by inserting a $U$ step at positions $(2n,0),\ldots,(2n-s,s)$, with $s=n-t$, are precisely the label productions of $\Omega_{steady}$.
Indeed, by inserting the $U$ step at positions $(2n,0),\ldots,(2n-(r-1),r-1)$ the edge line does not change and the paths obtained have labels $(h+k-1,2),\ldots,(h+1,k)$, respectively. Whereas, by inserting the $U$ step at position $(2n-i,i)$, for every $r\leq i\leq s$, the edge line becomes (or remains) $y=x-2(n-i)$ and the path has label $(0,i+2)$. Thus, we obtain the labels $(0,k+1),\ldots,(0,h+k+1)$, which are the second line of the production of $\Omega_{steady}$, completing the proof.
\end{proof}
Figure~\ref{prud_baxter_growth} depicts the growth of a steady path of size $n$ with edge line $y=x-2$; for any path, the corresponding edge line is drawn. 

\begin{figure}[ht]
\begin{center}
\IfFileExists{prud_baxter_growth.pdf}{\includegraphics[width=\textwidth]{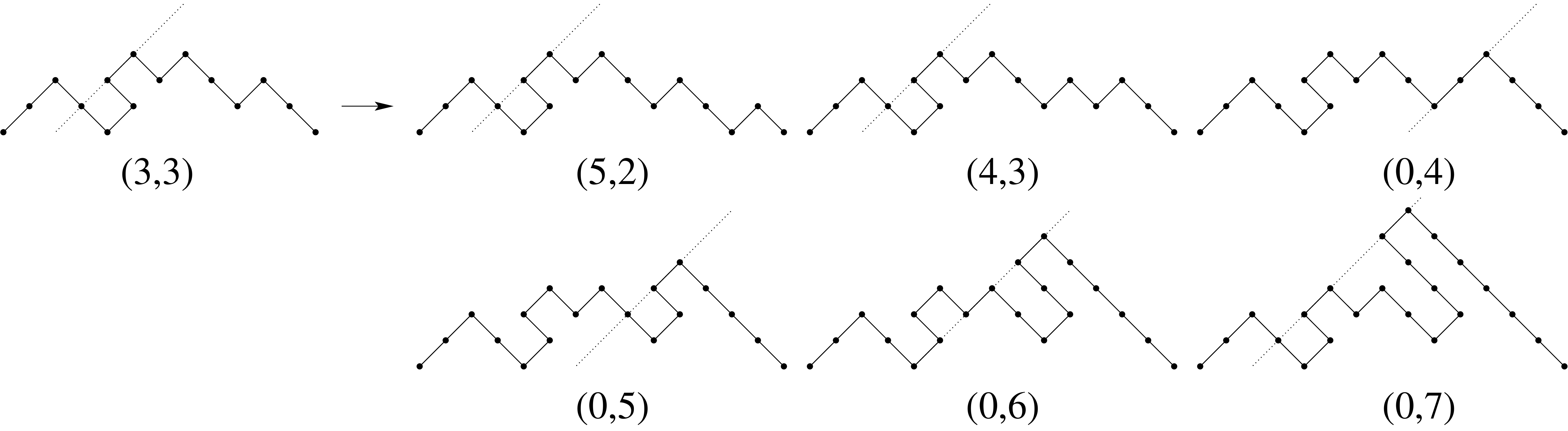}\vspace{-7mm}}{MISSING FILE}
\end{center}
\caption{The growth of a steady path according to rule $\Omega_{steady}$.}
\label{prud_baxter_growth}
\end{figure}
\subsection{Recursive bijection between steady paths and $AV(\callanpat)$}

Although at a first sight the succession rule $\Omega_{steady}$ does not resemble the rule $\Omega_{\callanpat}$ of Proposition~\ref{prop:1-23-4}, the following result follows by the fact that $\Omega_{steady}$ and $\Omega_{\callanpat}$ actually define the same generating tree.
\begin{proposition}\label{prop:1-23-4bij}
The number of steady paths of size $n$ is equal to the number of permutations in $AV_n(\callanpat)$, thus is the $n$-th powered Catalan number.
\end{proposition}
\begin{proof}
We prove the above proposition by showing that the succession rule $\Omega_{\callanpat}$ provided for the family $AV(\callanpat)$ is isomorphic to the rule $\Omega_{steady}$.

First, recall the production of the label $(1,k)$ according to $\Omega_{\callanpat}$, which appears as
\begin{equation}\label{production1k}
(1,k)\rightsquigarrow(1,k+1),(2,k),\ldots,(1+k,1)\,.\end{equation}
Using the same succession rule $\Omega_{\callanpat}$, the label $(h,0)$ produces according to
\begin{equation}\label{productionh0}(h,0)\rightsquigarrow(1,h),(2,h-1),\ldots,(h,1)\,.\end{equation}
Now, consider the generating tree defined by $\Omega_{\callanpat}$ and replace all the labels $(1,k)$ by $(k+1,0)$. According to the production~\eqref{production1k} the children of the node with a replaced label are
\[(k+1,0)\rightsquigarrow (k+2,0),(2,k),\ldots,(1+k,1)\,.\]
Setting $h=k+1$, this can be rewritten as 
\[(h,0)\rightsquigarrow (h+1,0),(2,h-1),\ldots,(h,1)\,,\]
which is exactly the production~\eqref{productionh0} after substituting $(1,h)$ for $(h+1,0)$ in it. 
Therefore, the substitution of all labels $(1,k)$ by $(k+1,0)$ allows us to rewrite the succession rule $\Omega_{\callanpat}$ as follows
\[\left\{\begin{array}{lll}
(2,0)&&\\
\\
(h,k) &\rightsquigarrow \hspace{-3mm}& (h+k+1,0),\\
& \hspace{-3mm}&(2,h+k-1), \ldots , (h,k+1),\\
& \hspace{-3mm}&(h+1,0), \ldots , (h+k,0). 
\end{array}\right.\]
It is straightforward to check that the growth provided for steady paths in Proposition~\ref{prop:steady_growth} defines the above succession rule by exchanging the interpretations of the two parameters $h$ and $k$ with respect to $\Omega_{steady}$.
\end{proof}

Having given two generating trees along the same succession rule for steady paths and permutations avoiding $\callanpat$, 
we deduce immediately a bijection between these classes. 
Namely, this bijection puts in correspondence objects of the two families according to their position in the associated generating tree. 
Of course, this bijection is not explicit, but recursive (following the way along which each object is built starting from the smallest one). 
We would of course wish for a nicer bijection between steady paths and permutations avoiding $\callanpat$. 
We did not succeed, but instead were able to provide a bijection between steady paths and permutations avoiding $\pat$ 
(which are also known to be enumerated by powered Catalan numbers). 

\subsection{One-to-one correspondence between steady paths and $AV(\pat)$
\label{sec:steadybijperm}}

\begin{thm}\label{thm:steadyperm}
There exists an explicit bijection between the family of steady paths and $AV(\pat)$.
\end{thm}
\begin{proof}
By Remark~\ref{rem:steady}, any steady path $P$ of size $n$ is uniquely determined by the positions of its up steps, namely by the points $(0,0)=(i_1,j_1),\ldots,(i_n,j_n)$.
These points that encode a unique steady path can in turn be encoded  from right to left by a sequence $(t_1,\ldots,t_n)$ of integers that records the Euclidean distance between these points and the main diagonal $y=x$. More precisely, the entry $t_k$ is the distance between the point $(i_{n+1-k},j_{n+1-k})$ and the line $y=x$, for any $1\leq k\leq n$. Note that $t_n=0$, because the point $(0,0)$ belongs to the main diagonal. Moreover, for any $1\leq k\leq n$, the entry $t_k$ is in the range $[0,n-k]$, since steady paths are constrained in the cone $\mathfrak{C}=\{(x,y)\in\mathbb{N}^2:y\leq x\}$. For instance, the steady path depicted in Figure~\ref{fig:prud}~(a) is encoded by the sequence $(5,3,0,4,1,0,1,0)$.
\smallskip

Then, we have that any steady path of size $n$ is defined by a particular sequence $(t_1,\ldots,t_n)$, for which $0\leq t_k\leq n-k$, for every $k$. Certainly, the set of all these particular sequences of size $n$ forms a subset of the set $\{\mathtt{T}(\pi):\pi\in\mathcal{S}_n\}$ of the left inversion tables of permutations of length $n$. Our aim is to prove that a left inversion table $t=(t_1,\ldots,t_n)$, with $0\leq t_k\leq n-k$, defines a steady path of size $n$ if and only if  $t\in\{\mathtt{T}(\pi):\pi\in AV_n(\pat)\}$.
\begin{itemize}
\item[$\Rightarrow)$] We prove the contrapositive. Suppose $t=(t_1,\ldots,t_n)$ is the left inversion table of a permutation $\pi\not\in AV_n(\pat)$. Then, since $\pi$ contains $\pat$, there must be three indices $i,j,\ell$, with $i<j<j+1<\ell$, such that $\pi_i<\pi_{\ell}<\pi_j<\pi_{j+1}$. Moreover, we can suppose without loss of generality that there are no points $\pi_s$ between $\pi_i$ and $\pi_{j}$ such that $\pi_s<\pi_i$. Otherwise, we could take $\pi_s\pi_j\pi_{j+1}\pi_\ell$ as our occurrence of $\pat$.

Then, by definition of the left inversion table $t=\mathtt{T}(\pi)$, since $\pi_j<\pi_{j+1}$ and $\pi_j>\pi_\ell$, we have that $0<t_j\leq t_{j+1}$. In addition, since there are no points $\pi_s$ between $\pi_i$ and $\pi_{j}$ such that $\pi_s<\pi_i$, and $\pi_j>\pi_\ell>\pi_i$, it holds that $t_i<t_j$. From this it follows that $t$ cannot encode a steady path $P$. Indeed, assuming such a path $P$ would exists, $t_i$ (resp. $t_j$, resp. $t_{j+1}$) must be the distance between the line $y=x$ and an up step $U^{(i)}$ (resp. $U^{(j)}$, resp. $U^{(j+1)}$), where $U^{(j+1)}$, $U^{(j)}$, and $U^{(i)}$ appear in this order from left to right. Since $t_{j+1}\geq t_{j}$, the up step $U^{(j)}$ must form either a $UU$ factor or $WU$ factor. Note that the line parallel to the main diagonal passing through $U^{(j)}$ cannot be $y=x$, since $t_j>0$. Let this line be $y=x-g$, with $g$ even positive number. Then, from $0\leq t_i<t_j$ it follows that the suffix of $P$ containing the up step $U^{(i)}$ exceeds the line $y=x-g$ passing through $U^{(j)}$.

\item[$\Leftarrow)$] Conversely, suppose for the sake of contradiction that there exists a left inversion table $t=(t_1,\ldots,t_n)$ which encodes a non-steady path $P$ of size $n$.

By definition of steady path, there must be in $P$ an up step $U^{(j)}$ not lying on the main diagonal such that it forms a factor $UU$ or $WU$, and an up step $U^{(i)}$, which is on the right of $U^{(j)}$, lying above the line parallel to $y=x$ and passing through $U^{(j)}$. This means that $0<t_j\leq t_{j+1}$, where $U^{(j+1)}$ is the up step which $U^{(j)}$ immediately follows, and $0\leq t_i<t_j$, with $i<j$. Thus, let $\pi=\mathtt{T}^{-1}(t)$. 
We have that $\pi_i<\pi_j<\pi_{j+1}$, and from $0\leq t_i<t_j$, there exists at least a point $\pi_\ell$, with $j<\ell$, such that $(\pi_j,\pi_\ell)$ is an inversion of $\pi$ and $(\pi_i,\pi_\ell)$ is not.
Consequently, $\pi_i\pi_j\pi_{j+1}\pi_\ell$ forms an occurrence of $\pat$.\qedhere
\end{itemize}
\end{proof}

\section{Bijection between steady paths and valley-marked Dyck paths} \label{sec:bijection}

We have exhibited (three but essentially) two succession rules for powered Catalan numbers: $\Omega_{steady}$ and $\Omega_{pCat}$. 
This leads us to classify powered Catalan structures into two groups: 
\begin{itemize}
\item those that appear as a rather simple generalization of Catalan structures, for which a growth according to the rule $\Omega_{pCat}$ can be found easily;
\item those that generalize Catalan structures, but for which a growth according to $\Omega_{pCat}$ is not immediate, and the parameter $k$ of Equation~\eqref{eq:reccallan} is not clearly understood.
\end{itemize}
Valley-marked Dyck paths are the emblem of the first group; while, steady paths as well as permutations avoiding $\callanpat$ rather belong to the second group of structures. 

We now take advantage of having a representative in each group which is a family of lattice paths confined to the region $\mathfrak{C}$ to provide a bijective link between the two groups, 
with Theorem~\ref{thm:bijection} below. 
Once this theorem will be proved, all powered Catalan structures involved in our study will be related as shown in Figure~\ref{fig:callanstructures}.

\begin{figure}[ht]
\centering
\begin{tikzpicture}[scale=.85]
\node at (0,3.2) {$AV({23}\text{-}1\text{-}4)$};
\node at (0,2.7) {\footnotesize Conjecture~\ref{con:shattuck}};
\draw (0,3) ellipse (1.9cm and .9cm);
\draw[thick,dashed,<->] (0,1.1)--(0,2);
\node at (0,.45) {Valley-marked};
\node at (0,0) {Dyck paths};
\node at (0,-.5) {\footnotesize Section~\ref{sec:mvpaths}};
\draw (0,0) ellipse (2cm and 1cm);
\draw[thick,<->] (2.1,0)--(3.9,0);
\node at (3,.4) {$\Omega_{pCat}$};
\draw[thick,<->] (0,-1.1)--(0,-1.9);
\node at (0.8,-1.55) {Thm.~\ref{thm:bijection}};
\node at (6,0.2) {$\mathbf{I}(=,>,>)$};
\node at (6,-.4) {\footnotesize Section~\ref{sec:callaninvseq}};
\draw (6,0) ellipse (1.9cm and .9cm);
\node at (-6.1,.35) {Increasing ordered};
\node at (-6.1,0) {trees};
\node at (-6.1,-.5) {\footnotesize Section~\ref{sec:mvpaths}};
\draw (-6.1,0) ellipse (2cm and 1cm);
\draw[thick,<->] (-2.2,0)--(-4,0);
\node at (-3,.4) {$\Omega_{pCat}$};
\node at (-6.1,-2.8) {$AV(\callanpat)$};
\node at (-6.1,-3.3) {\footnotesize Section~\ref{sec:2ndrulepcat}};
\draw (-6.1,-3) ellipse (1.8cm and .9cm);
\draw[thick,<->] (-6.1,-1.1)--(-6.1,-2);
\node at (-5.7,-1.6) {\cite{callan}};
\node at (0,-2.5) {Steady};
\node at (0,-3) {paths};
\node at (0,-3.5) {\footnotesize Section~\ref{sec:steadydef}};
\draw (0,-3) ellipse (2cm and 1cm);
\draw[thick,<->] (-2.1,-3)--(-4.2,-3);
\node at (-3.1,-2.6) {$\Omega_{steady}$};
\node at (6,-2.8) {$AV(\pat)$};
\node at (6,-3.4) {\footnotesize Section~\ref{sec:steadybijperm}};
\draw (6,-3) ellipse (1.75cm and .9cm);
\draw[thick,<->] (2.1,-3)--(4.2,-3);
\node at (3.1,-2.6) {Thm.~\ref{thm:steadyperm}};
\end{tikzpicture}\caption{All the structures known or conjectured to be enumerated by the powered  Catalan numbers and their relations: a solid-line arrow indicates a bijection (either recursive, or direct), while a dashed-line arrow indicates a conjectured bijection.
\label{fig:callanstructures}}
\end{figure}
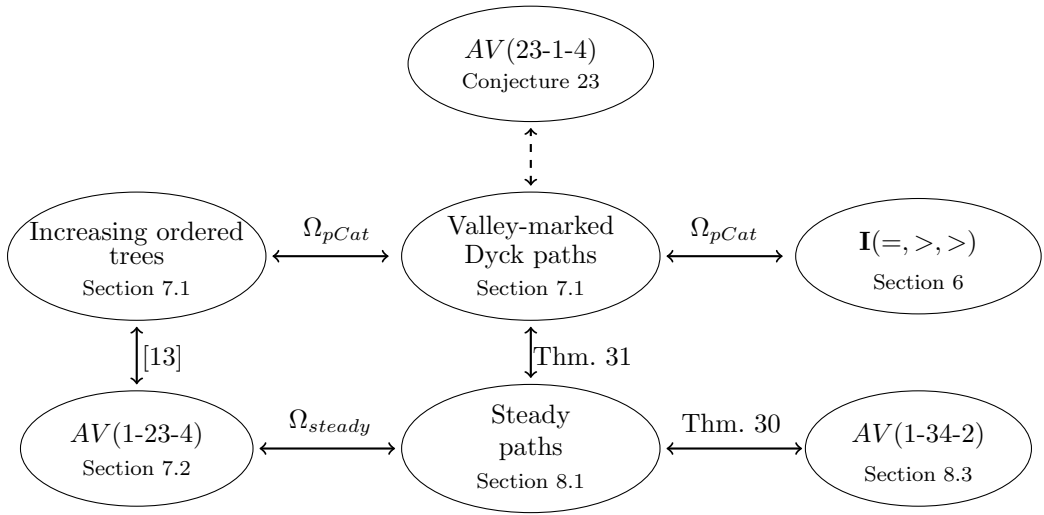

\begin{thm}\label{thm:bijection}
There is a size-preserving bijection between steady paths and marked-valley Dyck paths, 
which sends the number of $W$ steps to total mark, 
preserves the number of steps on the main diagonal, 
and sends the number of returns to the $x$-axis to the number of returns to the mark. 
\end{thm}

This bijection, hereafter denoted $\phi^*$, is explicit and is obtained as follows. Starting from a steady path, 
we apply to it a certain transformation $\phi$ that removes one $W$ step while increasing by one the total mark. 
We repeat this operation until all $W$ steps have been removed. 
To show that $\phi^*$ is a bijection, we actually describe its inverse. 
It is denoted $\theta^*$ and is also obtained by iterating a certain transformation $\theta$ (which is the inverse of $\phi$). 

For this strategy to work, we need to define our transformations on a family of paths having both $W$ steps and marks on their valleys. 
The specific family we consider is the following one. 

\begin{definition}
A \emph{valley-marked steady path} is a steady path $P$ where each valley receives a mark, according to the following conditions: 
\begin{description} 
 \item{(M1)} if $(i,k)$ pinpoints any valley of $P$, then a valley-marked steady path associated with $P$ must have a mark in a point $(i,j)$, where $0\leq j\leq k$;
 \item{(M2)} all valleys with nontrivial marks are (weakly) below all $W$ steps;
 \item{(M3)} a valley with a nontrivial mark may be at the same height as a $W$ step only if it appears to the right of the $W$ step. 
\end{description}
\end{definition}

The definition of total mark and return to the mark are extended to valley-marked steady paths in the obvious way. 

Clearly, valley-marked steady paths without $W$ steps are valley-marked Dyck paths, and 
valley-marked steady paths where all the marks of the valleys are trivial are (in obvious correspondence with) steady paths. 
Moreover, any of these three families of paths includes as a subfamily the classical Dyck paths (or a family of paths in obvious correspondence with them). 
More precisely, the set of Dyck paths is the intersection of the families of steady paths and valley-marked Dyck paths (where trivial marks are interpreted as inexistent).

\subsection{The transformation $\phi$ decreasing the number of $W$ steps}

Consider a valley-marked steady path $P$, assumed to contain at least one $W$ step. 
Consider the rightmost among the bottommost $W$ steps. 
This $W$ being bottommost, it cannot be preceded by a $W$. It also cannot be preceded by a $D$, since $DW$ factors are forbidden. 
Therefore, it is preceded by a $U$. This $U$ can only be preceded by a $D$, otherwise breaking one of the conditions (S1) and (S2) defining steady paths. 
We have therefore identified that our $W$ is preceded by a valley (encircled in Figure~\ref{fig:phi}(a)) which we call \emph{circled valley}. 
If $(i,k)$ pinpoints the circled valley, we call $k$ the height of this valley, and we denote by $h \leq k$ the height of its mark. 

We decompose our valley-marked steady path $P$ around this factor $DUW$. 
The $D$ step has a matching step to its left, which has to be a $U$ step, since $W$ has been chosen bottommost. 
The $U$ and the $W$ have matching $D$ steps to their right. 
Our path $P$ is therefore decomposed as 
\[
Pr \cdot U \cdot A \cdot DUW \cdot B \cdot D \cdot C \cdot D \cdot S,
\textrm{ see Figure \ref{fig:phi}(a),} 
\]
where $Pr$ (resp. $S$) is a prefix (resp. suffix) of the path, 
and $A$, $B$ and $C$ are factors of this path never going below their starting ordinate. 
Additionally, $B$ must be non-empty (since $WD$ factors are forbidden). 
Moreover, all valleys in $A$ or $B$ have a trivial mark, by conditions (M2) and (M3). 
And similarly, the only valleys in $C$ with nontrivial marks (if any) are at the ``ground level'' of $C$ (\emph{i.e.}, at height $k+1$). 

We define the image by $\phi$ of this path depending on whether $A$ is empty or not. 

\begin{itemize}
 \item If $A=\emptyset$, its image is 
\[
Pr \cdot U \cdot B \cdot UD \cdot C \cdot D \cdot S,
\textrm{ see Figure \ref{fig:phi}(b).}
\]
Note that in this case, a valley between $B$ and the successive $U$ step has been created, replacing the circled valley. 
 \item If $A\neq \emptyset$, its image is 
\[
Pr \cdot U \cdot A \cdot U \cdot B \cdot D \cdot C \cdot D \cdot S,
\textrm{ see Figure \ref{fig:phi}(c).} 
\]
Note that in this case, a valley between $A$ and the successive $U$ step has been created, replacing the circled valley. 
\end{itemize}
Performing this transformation, the circled valley has been moved to a new valley at height $k+1$, whose mark is set to $h+1$. 
Every other valley (even if it is moved by $\phi$) keeps its mark unchanged. 

\begin{figure}[ht]
\begin{center}
\IfFileExists{fig_phi.pdf}{\includegraphics[width=.8\textwidth]{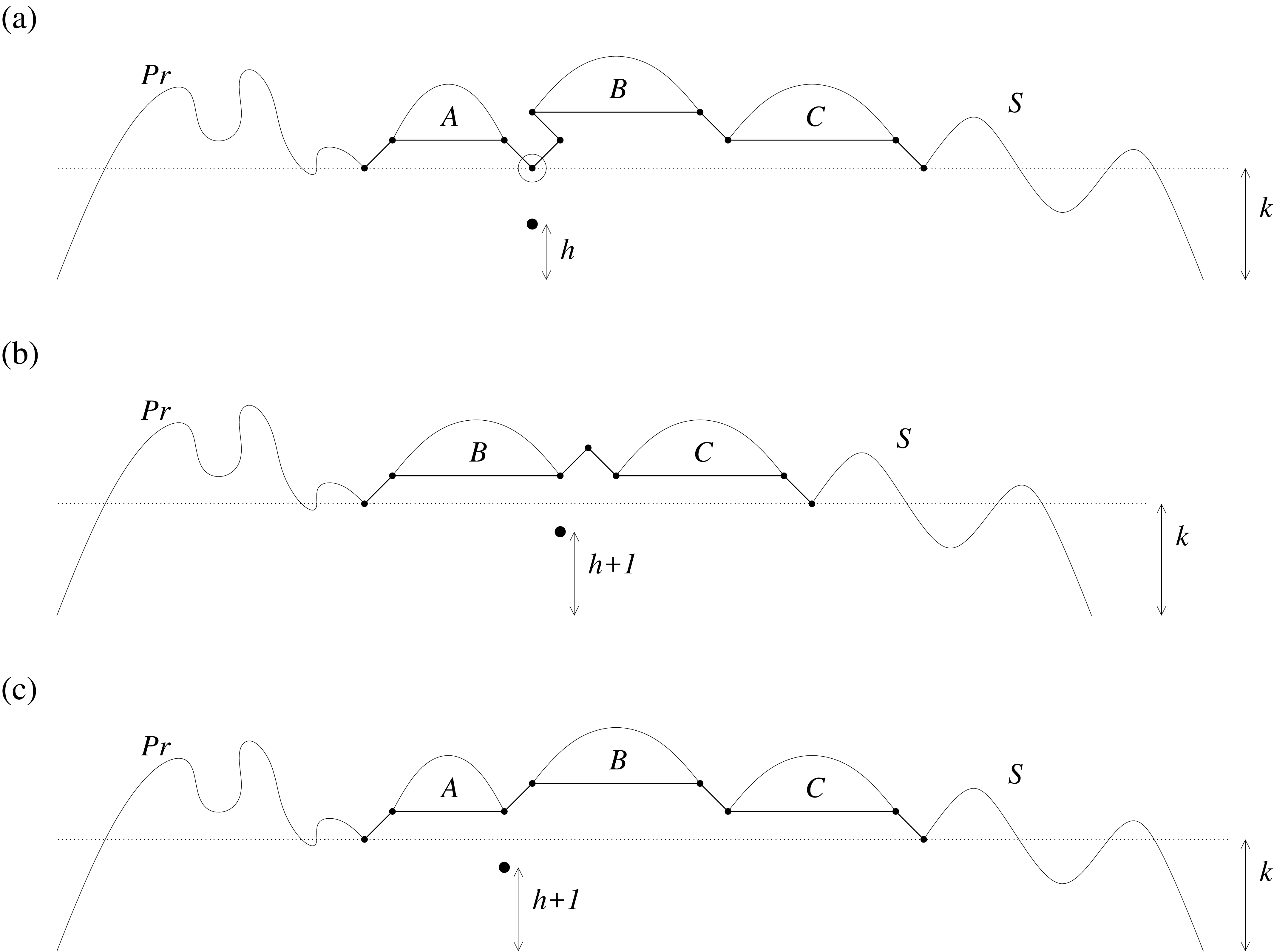}}{MISSING FILE}
\end{center}
\caption{(a) A valley-marked steady path assumed to contain at least one $W$ step;
(b) its image by $\phi$ in the case $A$ is empty; 
(c) its image by $\phi$ in the case $A$ is not empty.}
\label{fig:phi}
\end{figure}

\begin{lemma}
\label{lem:transformation_phi}
Let $P$ be a valley-marked steady path with at least one $W$ step, 
with total mark $m$, $k$ $W$ steps, $r$ returns to the mark, and $d$ steps on the main diagonal. 
It holds that
\begin{itemize}
 \item[i)] $\phi(P)$ is a valley-marked steady path. 
 \item[ii)] The total mark of $\phi(P)$ is $m+1$.
 \item[iii)] The number of $W$ steps in $\phi(P)$ is $k-1$. 
 \item[iv)] $\phi(P)$ has $d$ steps on the main diagonal. 
 \item[v)] $\phi(P)$ has $r$ returns to the mark. 
 \item[vi)] The valley that has been moved w.r.t. $P$ is the leftmost among the topmost valleys of $\phi(P)$ carrying a nontrivial mark. 
\end{itemize}
\end{lemma}

\begin{proof}
Items ii) and iii) are clear. 

Item v) is also very easy, since the only valley whose height is modified (by $+1$) has its mark modified accordingly (also by $+1$). 

To prove iv), note that in case $A$ is not empty, the only modification between $P$ and $\phi(P)$ is
that the considered $DUW$ factor is replaced by $U$. Neither of these two $U$ steps may be on the main diagonal (since $A$ is not empty), 
so $P$ and $\phi(P)$ have the same number of steps on the main diagonal. 
On the other hand, if $A$ is empty, $B$ is moved one cell to the left and one cell down w.r.t. the prefix $Pr \cdot U$ of $P$ and the rest of the path is unchanged. Since the $U$ step in the considered $DUW$ factor of $P$ is not on the main diagonal, 
and neither the $U$ step following $B$ in $\phi(P)$ (since $B$ is not empty), it follows that $P$ and $\phi(P)$ have the same number of steps on the main diagonal.

Items i) and vi) require more care. 
Consider first i). We need to check that $\phi(P)$ satisfies conditions (S1), (S2), (M1), (M2) and (M3). 

The case where $A$ is not empty is easier. 
Indeed, in this case, the relative positions of $Pr$, $A$, $B$, $C$ and $S$, 
as well as the lines supported by the $UU$ and $WU$ factors are unchanged. 
So, $\phi(P)$ satisfies conditions (S1) and (S2). 
For (M1), notice that all valleys except the circled one are not moved, and their marks are not changed, 
so we are just left with checking that the valley replacing the circled one satisfies condition (M1), 
which is immediate since both the height and the mark of this valley are increased by one. 
Because $P$ satisfies condition (M2), we know that any valley of $P$ with a nontrivial mark is at height at most $k+1$. 
The same stays true in $\phi(P)$ (the valley replacing the circled one having height $k+1$ exactly), 
ensuring that $\phi(P)$ also satisfies condition (M2). 
Finally, the rightmost among the bottommost $W$ steps (if any) of $\phi(P)$ 
either goes from height $k+j$ to $k+j+1$ with $j\geq 2$  
or goes from height $k+1$ to $k+2$ and is in the prefix $Pr \cdot U \cdot A$ of $\phi(P)$.
In the first case, since all valleys of $\phi(P)$ with a nontrivial mark are at height at most $k+1$ (from (M2)), it follows that condition (M3) is clearly satisfied. 
In the second case, since $P$ satisfies condition (M3), it must hold that any valley of $\phi(P)$ with a nontrivial mark at height $k+1$ either belongs to the suffix $B \cdot D \cdot C \cdot D \cdot S$ or is the replacement of the circled valley, and thus to the right of $Pr \cdot U \cdot A$. Then, condition (M3) is also satisfied by $\phi(P)$. 

We now consider the case where $A$ is empty. 
Although the path is modified more substantially, we note that the $U$ steps supporting a line parallel to $y=x$ that cannot be crossed to satisfy conditions (S1) and (S2) are the same in $P$ and $\phi(P)$. Next, we examine how $B$ and $C$ are moved. 
First, $B$ is moved one cell to the left and one cell down w.r.t. the prefix $Pr \cdot U$ of $P$, which is unchanged in $\phi(P)$. 
This makes sure that conditions (S1) and (S2) are not violated by the steps of $B$. 
Second, $C$ is moved one cell to the right and one cell up w.r.t. $B$. 
Noticing that at least one up step in $B$ supports a line parallel to $y=x$ imposing a condition to the suffix of the path, 
this makes sure that no step of $C$ (nor of the suffix following $C$) violates condition (S1) nor (S2). 
Third, the small peak that has been added between $B$ and $C$ clearly satisfies conditions (S1) and (S2). 
That (M1) is satisfied is clear, since again the only valley which is changed is the circled one, 
for which both the height and the mark are increased by one. 
To see that (M2) and (M3) are satisfied, 
observe first that the circled valley does not violate them, 
being either above the rightmost among the bottommost $W$ steps of $\phi(P)$, 
or at the same height but to its right. 
For all other valleys, it is enough to notice that all valleys in $B$ have trivial marks (since $P$ satisfies (M2)) 
and that the rightmost among the bottommost $W$ steps of $\phi(P)$ is 
either higher than that of $P$ or at the same height but to its left. 

We now turn to the proof of vi). 
First, observe that the valley that has been moved w.r.t. $P$ (the replacement of the circled valley)
has mark $h+1$ so is nontrivial. 
Note, in addition, that its height is $k+1$. 
Consider next another valley with a nontrivial mark in $\phi(P)$. It must correspond to a valley with a nontrivial mark in $P$. 
Since $P$ satisfies (M2) and (M3), such a valley may either be at height at most $k$ or it may be at height exactly $k+1$ but to the right of the considered $W$ step of $P$. 
In the case where $A$ is not empty, it follows immediately that the replacement of the circled valley is the leftmost among the topmost valleys of $\phi(P)$ with a nontrivial mark. On the other hand, if $A$ is empty, we have to use in addition that $B$ contains no valley with a nontrivial mark (which follows from (M2) on $P$). In both cases, we obtain that $\phi(P)$ satisfies vi). 
\end{proof}

\subsection{The transformation $\theta$ decreasing the total mark}

Consider a valley-marked steady path $P$, whose total mark is assumed to be non-zero. 
Among the valleys of $P$ having a nontrivial mark, choose the leftmost among the topmost ones. 
Denote by $k>0$ the height of this valley, and by $h>0$ its mark. 

Decompose $P$ around this marked valley $DU$ as follows. 
Let $A$ be the longest factor of $P$ ending with this $D$ step and which stays (weakly) above height $k$. 
(Necessarily, $A$ is not empty.)
Let $Pr$ be the prefix of $P$ before $A$. Note that the last step of $Pr$ may be a $U$ or a $W$ step going from height $k-1$ to height $k$, 
but because of condition (M2), it has to be a $U$. 
Let $B$ be the factor of $P$ between the $U$ step of the marked valley we consider and its matching $D$ step. 
Let $C$ be the longest factor of $P$ following this $D$ which stays (weakly) above height $k$. 
Note that $C$ is followed by a $D$ step. Call $S$ the suffix of $P$ after this $D$ step. 

It results that $P$ is decomposed as 
\[
Pr \cdot A \cdot U \cdot B \cdot D \cdot C \cdot D \cdot S,
\textrm{ see Figure \ref{fig:theta}(a). } 
\]

We define the image of $P$ by $\theta$ according to whether $B$ is empty or not. 
\begin{itemize}
 \item If $B = \emptyset$, $\theta(P)$ is the path 
\[
Pr \cdot D \cdot U \cdot W \cdot A \cdot D \cdot C \cdot D \cdot S,
\textrm{ see Figure \ref{fig:theta}(b). } 
\]
 \item If $B \neq \emptyset$, $\theta(P)$ is the path 
\[
Pr \cdot A \cdot D \cdot U \cdot W \cdot B \cdot D \cdot C \cdot D \cdot S,
\textrm{ see Figure \ref{fig:theta}(c). } 
\] 
\end{itemize}
In both cases, the considered valley of $P$ has been replaced in $\theta(P)$ by the one 
inside the created $DUW$ factor, which is at height $k-1$.
We set its mark to be $h-1$. 
All other valleys keep their marks unchanged.

\begin{figure}[ht]
\begin{center}
\IfFileExists{fig_theta.pdf}{\includegraphics[width=.8\textwidth]{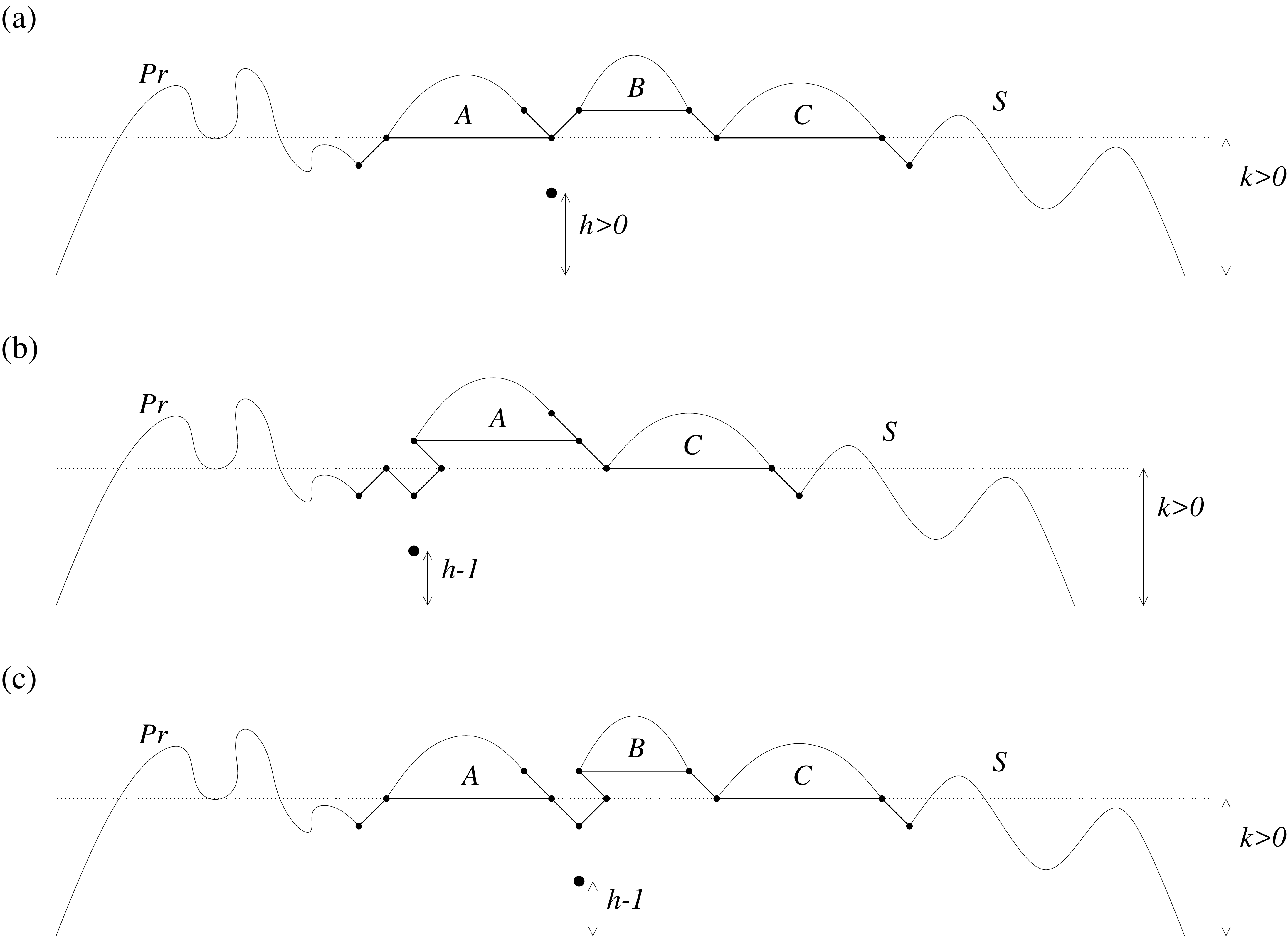}}{MISSING FILE}
\end{center}
\caption{(a) A valley-marked steady path with non-zero total mark;
(b) its image by $\theta$ in the case $B$ is empty; 
(c) its image by $\theta$ in the case $B$ is not empty.}
\label{fig:theta}
\end{figure}

\begin{lemma}
\label{lem:transformation_theta}
Let $P$ be a valley-marked steady path of total mark at least one, 
with total mark $m$, $k$ $W$ steps, $r$ returns to the mark, and $d$ steps on the main diagonal. 
It holds that
\begin{itemize}
 \item[i)] $\theta(P)$ is a valley-marked steady path. 
 \item[ii)] The total mark of $\theta(P)$ is $m-1$.
 \item[iii)] The number of $W$ steps in $\theta(P)$ is $k+1$. 
 \item[iv)] $\theta(P)$ has $d$ steps on the main diagonal. 
 \item[v)] $\theta(P)$ has $r$ returns to the mark. 
 \item[vi)] The $W$ step that has been added to $P$ is the rightmost among the bottommost $W$ steps of $\theta(P)$. 
\end{itemize}
\end{lemma}

\begin{proof}
Items ii), iii) and v) are clear. Item iv) is also rather easy. 
In particular, if $B$ is not empty, it follows because  all steps of $P$ and $\theta(P)$ are in the exact same places and the $U$ step of the modified valley may never lie on the main diagonal. 
In the case where $B$ is empty, $A$ is moved one cell to the right and one cell up w.r.t. the prefix $Pr$, while $C$, $Pr$ and $S$ are not moved, 
making sure that $P$ and $\theta(P)$ have the same number of steps on the main diagonal. 

As in the proof of Lemma~\ref{lem:transformation_phi}, the main part of the proof is to show i) and vi). 

In the case where $B$ is not empty, conditions (S1) and (S2) are clearly preserved 
(since the $U$ steps supporting the lines not to be crossed are the same, and $Pr$, $A$, $B$, $C$ and $S$ all stay at the same place). 
In the case where $B$ is empty, the lines not to be crossed also remain the same. 
Moreover, $Pr$, $C$ and $S$ stay at the same place
while $A$ is moved, but parallelly to the main diagonal. So, (S1) and (S2) stay satisfied. 

Condition (M1) obviously stays satisfied, since $\theta$ modifies the valleys only by moving one of them one level down, together with its mark. 
Since the chosen valley of $P$ is leftmost among the topmost valleys with nontrivial marks, 
it holds that the valleys of $A$ and $B$ (if any) all have trivial marks, 
and that the only valleys of $C$ with nontrivial marks (if any) are at ``ground level'' for $C$, \emph{i.e.}, at height $k$. 
This ensures that $\theta(P)$ satisfies (M2) and (M3) (in both cases $B=\emptyset$ and $B \neq \emptyset$). 

We are just left with the proof of vi). 
Recall that the chosen valley of $P$ is at height $k$ and has a nontrivial mark. 
By conditions (M2) and (M3), this implies that all $W$ steps of $P$ go from height $k$ to $k+1$ or are higher. 
Moreover, if $P$ has a $W$ step from height $k$ to $k+1$, it has to be in $Pr$ or $A$. 
This easily ensures vi).  
\end{proof}

\subsection{Proof of Theorem~\ref{thm:bijection}}

Let us denote by $\phi^*$ (resp. $\theta^*$) the transformation that takes a steady path (resp. 
marked-valley Dyck path) and iteratively applies $\phi$ (resp. $\theta$) to it as long as the path has some $W$ step 
(resp. positive total mark). To complete the proof of Theorem~\ref{thm:bijection}, 
we just have to show the following: 

\begin{thm}
$\phi^*$ is a size-preserving bijection between steady paths and marked-valley Dyck paths, 
whose inverse is $\theta^*$. 
Moreover, $\phi^*$ sends the number of $W$ steps to total mark, 
preserves the number of steps on the main diagonal, 
and sends the number of returns to the $x$-axis to the number of returns to the mark. 
\end{thm}

\begin{proof}
That $\phi^*$ applied to a steady path produces a marked-valley Dyck path follows immediately 
from the remark that a marked-valley Dyck path is just a marked-valley steady path with no $W$ step. 
Similarly, $\theta^*$ applied to a marked-valley Dyck path produces a steady path. 

To prove that $\phi^*$ is a bijection, we simply note that $\theta^*$ is its inverse. 
This is an immediate consequence of the fact that $\phi$ and $\theta$ are inverse of each other. 
This last claim follows by construction and items vi) of Lemmas~\ref{lem:transformation_phi} and~\ref{lem:transformation_theta}. 

From Lemma~\ref{lem:transformation_phi}, $\phi$ (and hence $\phi^*$) preserves the statistics 
``total mark + number of $W$ steps''. This implies that $\phi^*$ sends the number of $W$ steps to total mark. 
It follows similarly from Lemma~\ref{lem:transformation_phi} that $\phi^*$ preserves the number of steps on the main diagonal, 
and sends the number of returns to the $x$-axis to the number of returns to the mark. 
\end{proof}

A noticeable and nice property of the bijection $\phi^*$ is that it is the identity on the set of Dyck paths 
(interpreted either as valley-marked Dyck paths with only trivial marks, or as steady paths with no $W$ step). 

\subsection*{Acknowledgments}
We would like to thank Carla Savage and an anonymous referee for their comments which brought to our attention 
the papers~\cite{LinKernel,Yan} and the updated version of~\cite{savage}. 

We also express our gratitude to the editors of TCS, who made it possible to include new material 
(all of Section~\ref{sec:bijection}) and an additional author during the revision of this paper.

The first author was supported by the Australian Research Council grant DE170100186.


\begin{thebibliography}{99}
\bibitem{babson}
E. Babson and E. Steingr\'imsson,
\newblock\emph{Generalized Permutation Patterns and a Classification of the Mahonian Statistics}, S{\'e}minaire Lotharingien de Combinatoire,  B44b, 2000.

\bibitem{GFGT} 
C. Banderier, M. Bousquet-M\'elou, A. Denise, P. Flajolet, D. Gardy,
D. Gouyou-Beauchamps,
\newblock \emph{Generating functions for generating trees},
Disc. Math., 246:29--55, 2002.

\bibitem{Eco} 
E. Barcucci, A. Del Lungo, E. Pergola, R. Pinzani,
\newblock \emph{ECO: a methodology for the Enumeration of
Combinatorial Objects},
J. Diff. Eq. and App., 5:435--490, 1999.

\bibitem{larabaxter}
A.~M.~Baxter, L.~Pudwell, \newblock \emph{Enumeration schemes for vincular patterns}, Disc. Math., 312:1699--1712 , 2012.

\bibitem{shattuck} A.~M.~Baxter, M.~Shattuck, \newblock \emph{Some Wilf-equivalences for vincular patterns}, Journal of Combinatorics, 6:19--45, 2015.

\bibitem{slicings} 
N.~R. Beaton, M. Bouvel, V. Guerrini, S. Rinaldi, 
\newblock \emph{Slicings of parallelogram polyominoes: Catalan, Schr\"oder, Baxter, and other sequences},
arXiv preprint, Version 3, \url{https://arxiv.org/abs/1511.04864}, 2018. 

\bibitem{arxivV1} 
N.~R. Beaton, M. Bouvel, V. Guerrini, S. Rinaldi, 
\newblock \emph{Enumerating five families of pattern-avoiding inversion sequences; and introducing the powered Catalan numbers},
arXiv preprint, \url{https://arxiv.org/abs/1808.04114v1} 
(first version of the current paper), 2018.

\bibitem{BM}  
M. Bousquet-M\'elou,
\newblock \emph{Four classes of pattern-avoiding permutations under
one roof: generating trees with two labels},
Electron. J. Combin., 9(2), article R19, 2003.


\bibitem{BB} 
M. Bousquet-M\'elou, S. Butler,
\newblock \emph{Forest-like permutations},
Ann. Combin., 11:335--354, 2007.

\bibitem{MBM_Xin} 
M. Bousquet-M\'elou, G. Xin, \newblock \emph{On partitions avoiding 3-crossings}, S{\'e}minaire Lotharingien de Combinatoire, 54:B54e, 2005.

	 
\bibitem{semibaxterlong} 
M. Bouvel, V. Guerrini, A. Rechnitzer, S. Rinaldi, 
\newblock  \emph{Semi-Baxter and strong-Baxter: two relatives of the Baxter sequence}, 
To be published in SIDMA (SIAM J. Discrete Math.), 2018.

\bibitem{paper1} 
M. Bouvel, V. Guerrini, S. Rinaldi, 
\newblock \emph{Slicings of parallelogram polyominoes, or how Baxter and Schr\"oder can be reconciled}, 
Proceedings of FPSAC 2016, DMTCS proc. BC, 287--298, 2016.

\bibitem{callan} D. Callan, \newblock\emph{A bijection to count $(1\text{-}23\text{-}4)$-avoiding permutations}, arXiv preprint, \url{https://arxiv.org/abs/1008.2375}, 2010.

\bibitem{callan2} D.~Callan, \newblock\emph{A combinatorial survey of identities for the double factorial}, arXiv preprint, \url{https://arxiv.org/abs/0906.1317}, 2009.


\bibitem{corteel} 
S. Corteel, M. A. Martinez, C. D. Savage, M. Weselcouch, \newblock  \emph{Patterns in
inversion sequences I}, Discrete Mathematics and Theoretical Computer Science, 18(2), 2016.

\bibitem{CGHK78} 
F.R.K. Chung, R. Graham, V. Hoggatt, M. Kleiman,
\newblock \emph{The number of {B}axter permutations},
J. Comb. Theory A, 24(3):382--394, 1978.

\bibitem{duncan} P.~Duncan, E.~Steingrimsson, \newblock\emph{Pattern avoidance in ascent sequences}, Electron. J. Combin., 18(1), article P226, 2011.

\bibitem{1234sergi} S.~Elizalde, \newblock\emph{Asymptotic enumeration of permutations avoiding generalised patterns}, Advances in Applied Mathematics,
	36:138--155, 2006.

\bibitem{simoncino}
E.~Deutsch, E.~Munarini, S.~Rinaldi, \newblock\emph{Skew Dyck paths}. J. Statist. Plann. Inference, 140:2191--2203, 2010.

\bibitem{Flaj} 
P. Flajolet, R. Sedgewick,
\emph{Analytic Combinatorics}, Cambridge University Press, Cambridge, 2009.

\bibitem{TesiVeronica}
V. Guerrini, 
\emph{On enumeration sequences generalising Catalan and Baxter numbers}, 
PhD Thesis, University of Siena, 2017-2018. 

\bibitem{KM1}
E.J. Janse van Rensburg, T. Prellberg, A. Rechnitzer, \newblock\emph{Partially directed paths in a wedge}, J.
Comb. Theory A, 115:623--650, 2008.

\bibitem{KL} D.~Kim, Z.~Lin, \newblock \emph{Refined restricted inversion sequences}, {S{\'e}minaire Lotharingien de Combinatoire - FPSAC 2017}, 78B:52,
 2017.

\bibitem{LinKernel} 
Z.~Lin,
\newblock \emph{Restricted inversion sequences and enhanced $3$-noncrossing partitions}, arXiv preprint \url{https://arxiv.org/abs/1706.07213}, 2017.

\bibitem{mansour}
T. Mansour, M. Shattuck, \newblock \emph{Pattern avoidance in inversion sequences}, Pure Mathematics
and Applications, 25(2):157--176, 2015.

\bibitem{savage} 
M. A. Martinez, C. D. Savage, 
\newblock \emph{Patterns in Inversion Sequences II: Inversion Sequences Avoiding Triples of Relations}, 
Journal of Integer Sequences, Vol. 21, Article 18.2.2, 2018.

\bibitem{OEIS}
OEIS Foundation Inc.,
\newblock \emph{The On-line Encyclopedia of Integer Sequences},
\verb+http://oeis.org+, 2011.

\bibitem{zeilberger} 
M. Petkovsek, H.S. Wilf, D. Zeilberger,
\emph{A=B}, AK Peters, Wellesley, 1996.


\bibitem{Pudwell} 
L. Pudwell, \newblock \emph{Enumeration schemes for permutations
avoiding barred patterns},
Electron. J. Combin., 17(1), article R29, 2010.

\bibitem{SimonePP}
S. Rinaldi, 
\newblock \emph{Inversion sequences and generating trees}, 
Talk at the conference \emph{Permutation Patterns 2017}, June 26-30, Reykjavik. 
Abstract available at \url{https://pp2017.github.io/assets/pdf/abstracts/talks/simone_rinaldi.pdf}.

\bibitem{West_gt} 
J. West,
\newblock \emph{Generating trees and the Catalan and Schr\"oder numbers},
Disc. Math., 146:247--262, 1995.

\bibitem{Yan}
S.~H.~F. Yan, 
\newblock \emph{Bijections for inversion sequences, ascent sequences, and $3$-nonnesting set partitions}, 
Applied Mathematics and Computation, 325:24-30, 2018.

\bibitem{cretele} D.~Zeilberger, \newblock\emph{The method of creative telescoping}, Journal of Symbolic Computation, 11:195--204, 1991.

\bibitem{ZeilbergerCAT} D.~Zeilberger, \newblock\emph{The umbral transfer-matrix method: I. Foundations}, {J. Combin. Theory Ser. A},
	91:451--463, 2000.

\end{thebibliography}
\end{document}